\documentclass[a4paper,11pt,oneside]{article}
\usepackage[margin=1in]{geometry}
\pdfoutput=1

\RequirePackage[backref=page]{hyperref}
\usepackage{amsmath, amsthm, amssymb, natbib, graphicx, url, algorithm2e}
\usepackage{mathtools}

\numberwithin{equation}{section}
\hypersetup{
	colorlinks=true,
  linkcolor=red,          %
  citecolor=blue,         %
  filecolor=magenta,      %
  urlcolor=cyan           %
}

 \renewcommand*{\backrefalt}[4]{%
    \ifcase #1%
     \or (page:~#2)%
     \else (pages:~#2)%
    \fi%
    }

\newcommand{\pred}[1]{\boldsymbol{1}\left\{#1\right\}}
\newcommand{\R}{\mathbb{R}}
\newcommand{\C}{\mathbb{C}}
\newcommand{\F}{\mathbb{F}}
\newcommand{\N}{\mathbb{N}}

\newcommand{\unit}{\boldsymbol{1}}
\newcommand{\identity}{\boldsymbol{I}}
\newcommand{\zero}{\boldsymbol{0}}
\newcommand{\trn}{^\intercal}

\DeclareMathOperator{\diag}{diag}
\DeclareMathOperator{\Tr}{Tr}
\newcommand{\expo}[1]{\exp\left( #1 \right)}

\newcommand{\abs}[1]{\left| #1 \right|}
\newcommand{\paren}[1]{\left( #1 \right)}
\newcommand{\nrm}[1]{\left\Vert #1 \right\Vert}
\newcommand{\vertiii}[1]{{\left\vert\kern-0.25ex\left\vert\kern-0.25ex\left\vert #1 
    \right\vert\kern-0.25ex\right\vert\kern-0.25ex\right\vert}}
\newcommand{\tv}[1]{\nrm{#1}_{\mathsf{TV}}}

\DeclarePairedDelimiter\ceil{\lceil}{\rceil}
\DeclarePairedDelimiter\floor{\lfloor}{\rfloor}

\newcommand{\M}{\mathcal{M}}
\newcommand{\mc}{\boldsymbol{M}}
\newcommand{\bpi}{\boldsymbol{\pi}}
\newcommand{\bmu}{\boldsymbol{\mu}}

\newcommand{\estmc}{\widehat{\mc}} %
\newcommand{\estbpi}{\widehat{\bpi}}
\newcommand{\estpi}{\widehat{\pi}}

\newcommand{\sg}{\gamma}
\newcommand{\asg}{\gamma_\star}
\newcommand{\pssg}{\sg_{\mathsf{ps}}}
\newcommand{\estpssg}{\widehat{\sg}_{\mathsf{ps}}}

\newcommand{\Dpi}{\boldsymbol{D}_{\bpi}}
\newcommand{\Lsym}{\boldsymbol{L}}

\newcommand{\tmix}{t_{\mathsf{mix}}}

\newcommand{\Smc}{\mathcal{S}_d}
\newcommand{\smc}{\boldsymbol{S}}

\newcommand{\X}{\boldsymbol{X}}
\newcommand{\PR}[2][]{\mathbf{P}_{#1}\left( #2 \right)}

\newcommand{\E}[2][]{\mathbf{E}_{#1}\left[ #2 \right]}
\newcommand{\Var}[2][]{\mathbf{Var}_{#1} \left[ #2 \right]}
\newcommand{\dist}{\boldsymbol{D}}

\newcommand{\eps}{\varepsilon}

\newcommand{\ub}{_{\textrm{{\tiny \textup{UB}}}}}
\newcommand{\lb}{_{\textrm{{\tiny \textup{LB}}}}}

\newcommand{\G}{\mathcal{G}}

\newcommand{\kl}[2]{D_{\textrm{\textup{KL}}}
  \left(#1 \middle| \middle| #2\right)}

\newcommand{\beq}{\begin{eqnarray*}}
\newcommand{\eeq}{\end{eqnarray*}}
\newcommand{\beqn}{\begin{eqnarray}}
\newcommand{\eeqn}{\end{eqnarray}}

\newcommand{\Y}{\boldsymbol{Y}}
\newcommand{\W}{\boldsymbol{W}}

\newcommand{\estasg}{\widehat{\gamma}_\star}
\newcommand{\estDpi}{\widehat{\boldsymbol{D}}_{\bpi}}

\newcommand{\EE}[1]{\mathbf{\mathcal{E}}_{#1}}

\newcommand{\pssgK}{\sg_{\mathsf{ps} [K]}}
\newcommand{\estpssgK}{\widehat{\sg}_{\mathsf{ps} [K]}}

\newcommand{\Nmax}{N_{\max}}
\newcommand{\Nmin}{N_{\min}}
\newcommand{\bu}{\boldsymbol{u}}
\newcommand{\bv}{\boldsymbol{v}}

\newcommand{\estginv}{\widehat{\boldsymbol{A}}^{\#}}
\newcommand{\kps}{{k_{\mathsf{ps}}}}

\newcommand{\eqdef}{\doteq}
\newcommand{\eqset}{{:=}}
\newcommand{\compt}[1]{\overline{#1}}
\newcommand{\pimin}{\pi_\star}
\newcommand{\estpimin}{\hat{\pi}_\star}
\newcommand{\rev}{^{\dagger}}
\newcommand{\bigO}{\mathcal{O}}
\newcommand{\hide}[1]{}
\newcommand{\col}{{\textrm{{\tiny \textup{col}}}}}
\newcommand{\row}{{\textrm{{\tiny \textup{row}}}}}
\newcommand{\mydiv}[2]{{#2}\equiv0\pmod{#1}}

\newtheorem{theorem}{Theorem}[section]
\newtheorem{lemma}{Lemma}[section]
\newtheorem{corollary}{Corollary}[section]
\newtheorem{remark}{Remark}[section]
\newcommand{\set}[1]{\left\{ #1 \right\}}

\title{Estimating the Mixing Time of Ergodic Markov Chains}
\author{G. Wolfer, A. Kontorovich} %

\begin{document}

\maketitle

\begin{abstract}%
We address the problem of estimating the mixing time $\tmix$ of an arbitrary ergodic finite Markov chain from a single trajectory of length $m$. The reversible case was addressed by \citet{hsu-mixing-ext}, who left the general case as an open problem. In the reversible case, the analysis is greatly facilitated by the fact that the Markov operator is self-adjoint, and Weyl's inequality allows for a dimension-free perturbation analysis of the empirical eigenvalues. As \citeauthor{hsu-mixing-ext} point out, in the absence of reversibility (and hence, the non-symmetry of the pair probabilities matrix), the existing perturbation analysis has a worst-case exponential dependence on the number of states $d$. Furthermore, even if an eigenvalue perturbation analysis with better dependence on $d$ were available, in the non-reversible case the connection between the spectral gap and the mixing time is not nearly as straightforward as in the reversible case. Our key insight is to estimate the {\em pseudo-spectral gap} instead, which allows us to overcome the loss of self-adjointness and to achieve a polynomial dependence on $d$ and the minimal stationary probability $\pimin$. Additionally, in the reversible case, we obtain simultaneous nearly (up to logarithmic factors) minimax rates in $\tmix$ and precision $\eps$, closing a gap in \citeauthor{hsu-mixing-ext}, who treated $\eps$ as constant in the lower bounds. Finally, we construct fully empirical confidence intervals for the pseudo-spectral gap, which shrink to zero at a rate of roughly $1/\sqrt m$, and improve the state of the art in even the reversible case.
\end{abstract}

\section{Introduction}

We address the problem of estimating the mixing time $\tmix$ of
a
Markov chain with transition probability matrix $\mc$
from a single trajectory of observations.
Approaching the problem from a minimax perspective,
we construct point estimates
as well as fully empirical
confidence intervals for
$\tmix$ (defined in (\ref{definition:tmix})) of an unknown
{ergodic finite state time homogeneous} Markov chain.

It is a classical result \citep{levin2009markov} that
the mixing time of
an ergodic and reversible Markov chain
is controlled by its
\emph{absolute spectral gap} $\asg$
and minimum stationary probability $\pimin$:
\begin{equation}
\label{eq:control-tmix-reversible}
\left( \frac{1}{\asg} - 1 \right) \ln{2} \leq \tmix \leq\frac{\ln \left( 4/\pimin \right)}{\asg} ,
\end{equation}
which \citet{hsu2015mixing} leverage to estimate $\tmix$
in the reversible case.
For non-reversible Markov chains,
the relationship between the spectrum and the mixing time is not nearly as straightforward.
Any eigenvalue $\lambda\neq1$ provides a lower bound on the mixing time
\citep{levin2009markov},
\begin{equation}
\left( \frac{1}{1 - \abs{\lambda}} - 1\right)\ln{2} \leq \tmix,
\end{equation}
and upper bounds may be obtained in
terms of the spectral gap of the
\emph{multiplicative reversiblization}
\citep{fill1991eigenvalue, montenegro2006mathematical},
\begin{equation}
\tmix \leq  \frac{2}{\sg\left(\mc \rev \mc\right)} \ln\left( 2\sqrt{\frac{1 - \pimin}{\pimin}} \right).
\end{equation}
Unfortunately, the latter estimate is far from sharp\footnote{
  Consider the Markov chain on $\set{1, 2, 3}$
  with the transition probability matrix
  $\mc = \left(\begin{smallmatrix}0 & 1 & 0 \\ 0 & 0 & 1 \\ 1/2 & 0 & 1/2\end{smallmatrix}\right)$,
    which is rapidly mixing despite having $\sg(\mc \rev \mc ) = 0$ \citep{montenegro2006mathematical}.
  }.
A more delicate quantity, the \emph{pseudo-spectral gap}, $\pssg$
(formally defined in \eqref{definition:pssg})
was introduced by \citet{paulin2015concentration}, who showed that for ergodic $\mc$,
\begin{equation}
\label{eq:control-tmix-non-reversible}
  \frac{1}{2\pssg}\leq \tmix \leq \frac{1}{\pssg} \left(
  \ln \frac{1}{\pimin} 
  + 2 \ln 2 +1 \right).
\end{equation}
Thus, $\pssg$
plays a role analogous to that of $\asg$ in the reversible case,
in that it
controls the mixing time of non-reversible chains from above and below
--- and will be our main quantity of interest throughout the paper.

\section{Main results}

Here we give an informal overview of our main results, all of which pertain to
an unknown $d$-state ergodic time homogeneous Markov chain $\mc$
with mixing time $\tmix$ and minimum stationary probability $\pimin$.
The formal statements are deferred to 
Section~\ref{section:formal-results}. {\em Sample complexity} refers
to the trajectory length of observations drawn from $\mc$.

\begin{enumerate}
\item We determine the minimax sample complexity of estimating $\pimin$
  to within a relative error of $\eps$
  to be $\tilde{\Theta} \left( \frac{\tmix}{\eps^2 \pimin} \right)$. This improves
  the state of the art even in the reversible case.
\item
  We upper bound the sample complexity
  of estimating the \emph{pseudo-spectral gap} $\pssg$ of
  any ergodic $\mc$
  to within a absolute error of $\eps$
by
$\tilde \bigO \left(  \frac{\max \set{\tmix, \mathcal{C}(\mc)}}{\pimin \eps^2} \right)$,
where $1\le\mathcal{C}(\mc)\le d/\pimin$
captures a notion of how far
$\mc$
is
from being doubly-stochastic.
\item
  We lower bound the sample complexity
  of estimating $\tmix$ by
  $\tilde \Omega \left( \frac{\tmix d}{\eps^2} \right)$,
  which holds for both the reversible and non-reversible cases.
  This shows that our upper bound is sharp in $\eps$,
  up to logarithmic factors.
\item We construct fully empirical confidence intervals for $\pimin$ and $\pssg$
  without assuming reversibility.
\item Finally, our analysis
narrows the width of the confidence intervals and improves their
computational cost as compared to the state of the art.
\end{enumerate}

\section{Related work}

\sloppypar Our work is largely motivated by PAC-type learning problems with dependent data.
Many results from statistical learning and empirical process theory have been extended to
dependent data
with
sufficiently rapid mixing
\citep{yu1994rates, karandikar2002rates, gamarnik2003extension, mohri2008stability, mohri2009rademacher, steinwart2009fast, steinwart2009learning, me-cosma-nips, wolfer2018minimax}.
These have been used to
provide generalization guarantees that depend on the possibly unknown mixing properties of the process.
In the Markovian setting,
the relevant quantity is usually the mixing time,
and therefore
empirical estimates of this quantity yield corresponding data-dependent generalization bounds.

Other applications include MCMC diagnostics for non-reversible chains, which have recently gained interest through {accelerations} methods. Chains generated by the classical Metropolis-Hastings are reversible,
which
is instrumental in
analyzing the stationary distribution.
However, non-reversible chains may enjoy better {mixing properties} as well as improved asymptotic variance.
For theoretical and experimental results in this direction, see \citet{hildebrand1997rates, chen1999lifting, diaconis2000analysis, neal2004improving, sun2010improving, suwa2010markov, turitsyn2011irreversible,  chen2013accelerating, vucelja2016lifting}.

The problem of obtaining empirical estimates on the mixing time (with confidence)
was first addressed in \citet{hsu2015mixing} for the reversible setting,
by reducing the task to one of
estimating the absolute spectral gap and minimum stationary probability.
\citeauthor{hsu2015mixing}
gave
a point estimator for the {absolute spectral gap}, up to
fixed relative error,
with sample complexity
between
$\Omega \left(\frac{d}{\asg} + \frac{1}{\pimin} \right)$
and
$\tilde{\bigO} \left( \frac{1}{\asg^3 \pimin} \right)$.
Following up,
\citet{levin2016estimating}
sharpened the
upper bound of $\tilde{\bigO} \left( \frac{1}{\asg \pimin} \right)$, again leveraging properties of the {absolute spectral gap} of the unknown chain.
Additionally, \citet{hsu2015mixing}
presented a point estimator for
$\pimin$
with sufficient sample complexity
$\tilde{\bigO} \left( \frac{1}{\pimin \asg} \right)$.
The state of the art, as well as our improvements, are summarized in
Table~\ref{table:comparison}.

More recently, also in the reversible setting, algorithmic improvements in a different sampling model were obtained
\citep{qin2017estimating, combes2018computationally}, for which new estimation techniques,
derived from the \emph{power iteration method} were introduced; these
focus on computational rather than statistical efficiency.
Our stringent one-sequence sampling model is at the root of many of the technical challenges encountered;
allowing, for example, access to a 
\emph{restart mechanism} of the chain from any given state
considerably simplifies the estimation problem
\citep{batu2000testing, batu2013testing, bhattacharya2015testing}.

\section{Notation and definitions}

We define $[d] \eqdef \set{1,\ldots,d}$,
denote
the simplex of all distributions over $[d]$
by $\Delta_d$, and the collection of all $d\times d$ row-stochastic matrices by $\M_d$.
$\mathbb{N}$ will refer to $\set{1, 2, 3, \dots}$, and in particular
$0\notin\mathbb{N}$.
For $\bmu\in\Delta_d$, we will write either $\bmu(i)$ or $\mu_i$, as dictated by convenience. All vectors are rows unless indicated otherwise. A Markov chain on $d$ states being entirely specified by an initial distribution $\bmu\in\Delta_d$ and a row-stochastic transition matrix $\mc\in\M_d$, we identify the chain with the pair $(\mc,\bmu)$. Namely, by $(X_1,\ldots,X_m)\sim(\mc,\bmu)$, we mean that
\begin{equation}
\PR{(X_1,\ldots,X_m)=(x_1,\ldots,x_m)}=\bmu(x_1)\prod_{t=1}^{m-1}\mc(x_t,x_{t+1}).
\end{equation}
We write $\PR[\mc,\mu]{\cdot}$ to denote probabilities over sequences induced by the Markov chain $(\mc,\bmu)$,
and omit the subscript when it is clear from context. 

\paragraph{Skipped chains and associated random variables.}

For a Markov chain $X_1, \dots X_m \sim (\mc, \bmu)$, for any $k \in \mathbb{N}$ and $r \in \set{0, \dots, k - 1}$ we can define the $k$-skipped $r$-offset Markov chain, 
$$X_{r + k}, X_{r + 2k}, \dots, X_{r + t k}, \dots, X_{r + \floor{(m-r)/k} k} \sim (\mc^k, \bmu \mc^{r}),$$
and we will write it $X_1^{(k, r)}, \dots, X_{\floor{(m-r)/k}}^{(k, r)}$, or more simply $X_1^{(k)}, \dots, X_{\floor{m/k}}^{(k)}$ when $r = 0$. The main random quantities in use throughout this work are now defined for clarity. For two states $i$ and $j$, a skipping rate $k$, and an offset $r$, we define the \emph{number of visits to state $i$} to be
\begin{equation}
\label{equation:def-number-of-visits}
N_{i}^{(k, r)} \eqdef \abs{\set{1 \leq t \leq \floor{(m - r)/k} - 1: X_{t}^{(k, r)} = i}}, \quad N_{i}^{(k)} \eqdef N_{i}^{(k, 0)}
\end{equation}
and the \emph{number of transitions from $i$ to $j$} to be
\begin{equation}
\label{equation:def-number-of-transitions}
N_{i j}^{(k, r)} \eqdef \abs{\set{1 \leq t \leq \floor{(m-r)/k} - 1: X_t^{(k, r)} = i, X_{t+1}^{(k, r)} = j}}, \quad N_{i j}^{(k)} \eqdef N_{i j}^{(k, 0)}.
\end{equation}
We will also use the shorthand notation
\begin{equation}
\label{equation:def-random-variable-convenient-notation}
N_{i} \eqdef N_{i}^{(1)}, \ N_{i j} \eqdef N_{i j}^{(1)}, \ 
\Nmax^{(k)} \eqdef \max_{i \in [d]} N_i^{(k)}, \ \Nmin^{(k)} \eqdef \min_{i \in [d]} N_i^{(k)}, \ \Nmax \eqdef \Nmax^{(1)}, \ \Nmin \eqdef \Nmin^{(1)}.
\end{equation}

\paragraph{Stationarity.}

The Markov chain $(\mc,\bmu)$ is {\em stationary} if
$\bmu=\bmu\mc$ (i.e. $\bmu$ is a left-eigenvector associated to the eigenvalue $1$).
Unless noted otherwise, $\bpi$ is assumed to be a stationary distribution of the Markov chain in context.
We also define $\Dpi \eqdef \diag(\bpi)$, the diagonal matrix whose entries correspond to the stationary distribution, i.e. $\bpi = \unit \cdot \Dpi$, with
$\unit = (1, \dots, 1)$.

\paragraph{Ergodicity.}

The Markov chain $(\mc,\bmu)$ is {\em ergodic} if $\mc^k > 0$ (entry-wise positive) for some $k\ge1$. If $\mc$ is ergodic, it has a unique stationary distribution $\bpi$ and moreover $\pimin>0$, where $\pimin = \min_{i\in[d]}\pi_i$ is called the \emph{minimum stationary probability}.
We henceforth only consider ergodic chains.

\paragraph{Mixing time.}

When the chain is ergodic, we can define its mixing time as the number of steps it
requires to converge
to its stationary distribution within a constant precision (traditionally taken to be $1/4$):
\begin{equation}
\label{definition:tmix}
\tmix \eqdef \min_{t \in \N} \set{ \sup_{\bmu \in \Delta_{d}} \tv{\bmu \mc^{t-1} - \bpi} \leq \frac{1}{4} }
.
\end{equation}

\paragraph{Reversibility.}

A {\em reversible} $\mc\in\M_d$ satisfies {\em detailed balance} for some distribution $\bmu$: for all $i,j\in[d]$, $\mu_i\mc(i,j)=\mu_j\mc(j,i)$ --- in which case $\bmu$ is necessarily the unique stationary distribution. The eigenvalues of a reversible $\mc$ lie in $(-1,1]$, and these may be ordered (counting multiplicities): $1=\lambda_1\ge\lambda_2\ge\ldots\ge\lambda_d$. The \emph{spectral gap} and \emph{absolute spectral gap} are respectively defined to be 
\begin{equation}
\label{definition:asg}
\sg \eqdef 1-\lambda_2(\mc), \qquad \asg \eqdef 1 - \max \set{\lambda_2(\mc), \abs{\lambda_d(\mc)}}.\end{equation}
It follows from above that whenever $\mc$ is \emph{lazy} ($\forall i \in [d], \mc(i,i) \geq 1/2$), all eigenvalues of $\mc$ are positive and $\sg = \asg$. The matrix $\Dpi \mc$ consists of the \emph{doublet probabilities} associated with $\mc$ and is symmetric when $\mc$ is reversible.
The {\em rescaled transition matrix},
\beqn
\label{eq:Ldef}
\Lsym \eqdef \Dpi^{1/2}  \mc \Dpi^{-1/2},
\eeqn
is also symmetric for reversible $\mc$ (but not in general).
Since
$\Lsym$ and $\mc$ are similar matrices, their eigenvalue systems are identical.

\paragraph{Non-reversibility.}

Chains that do not satisfy the detailed balanced equations are said to be non-reversible. In this case, the eigenvalues
may be complex,
and the transition matrix may not be diagonalizable, even over $\C$.
\citet{paulin2015concentration} defines the {\em pseudo-spectral gap} by 
\begin{equation}
\label{definition:pssg}
\pssg \eqdef \max_{k \in \N} \set{ \frac{\sg((\mc \rev)^k \mc^k)}{k} },
\end{equation}
where $\mc \rev$ is the \emph{time reversal} of $\mc$, given by $\mc \rev (i,j) \eqdef \bpi(j)\mc(j,i)/\bpi(i)$ and $\sg$ is defined at (\refeq{definition:asg}); the expression $\mc \rev \mc$ is called the \emph{multiplicative reversiblization} of $\mc$.
The chain
$\mc \rev \mc$ is always reversible, and its eigenvalues are all real and non-negative  \citep{fill1991eigenvalue}. We also denote by $\kps$ the smallest positive integer
such that $\pssg = \frac{\sg \left( \left(\mc \rev \right)^\kps \mc ^\kps\right)}{\kps}$; this is
the power of $\mc$ for which the multiplicative reversiblization achieves\footnote{Note that for ergodic chains the pseudo-spectral is always achieved for finite $k$ so that $\kps$ is properly defined.  
Indeed, writing $g(k) \mapsto \sg((\mc^k)\rev \mc^k)/k$, it is a fact that $0 \leq g(k) \leq 1/k$.}
its pseudo-spectral gap. Intuitively, the pseudo-spectral gap is a generalization of the multiplicative reversiblization approach of \citeauthor{fill1991eigenvalue}, and for a reversible chain, the pseudo-spectral and absolute spectral gap are similar within a multiplicative factor of 2 (Lemma~\ref{lemma:pssg-reversible}).

\paragraph{Norms and metrics.}

We use the standard $\ell_1, \ell_2$ norms $\nrm{z}_p=\paren{\sum_{i\in[d]}|z_i|}^{1/p}
$;
in the context of distributions (and up to a convention-dependent factor of $2$),
$p=1$
corresponds to the total variation norm. For $A \in \R^{d\times d}$, define the spectral radius $\rho(A)$ to be the largest absolute value of the eigenvalues of $A$, and recall the following operator norms for real matrices,
\begin{equation}
\nrm{A}_\infty = \max_{i\in[d]}\sum_{j\in[d]}|A(i,j)|, \quad \nrm{A}_1 = \max_{j\in[d]}\sum_{i\in[d]}|A(i,j)|, \quad \nrm{A}_2 = \sqrt{\rho(A \trn A)}.
\end{equation}
We denote by $\langle \cdot, \cdot \rangle_{\bpi}$ the inner product on $\R^{d}$ defined by
$\langle \mathbf{f}, \mathbf{g} \rangle_{\bpi} \eqdef \sum_{i \in [d]} \mathbf{f}(i)\mathbf{g}(i)\bpi(i)$,
and
write $\nrm{\cdot}_{2, \bpi}$ for its associated norm; $\ell^2(\bpi)$ is the resulting Hilbert space.
To any
$(\mc,\bmu)$, we also associate
\begin{equation}
\nrm{\bmu/\bpi}_{2, \bpi}^2 \eqdef \sum_{i \in [d]} \mu_i^2/\pi_i,
\end{equation}
which
provides a notion of
``distance from  stationarity''
and
satisfies
$\nrm{\bmu/\bpi}_{2, \bpi} \le {1}/{\pimin}$.
For two $[d]$-supported distributions $\dist = (p_1, \dots, p_d)$ and $\dist' = (q_1, \dots, q_d)$, we also define respectively the Hellinger distance and the KL divergence,
\begin{equation}
H(\dist, \dist') = \frac{1}{\sqrt{2}} \sqrt{\sum_{i \in [d]} (\sqrt{p_i} - \sqrt{q_i})^2} \quad \kl{\dist}{\dist'} = \sum_{i \in [d]} p_i \ln\left( \frac{p_i}{q_i} \right).
\end{equation}

\paragraph{Asymptotic notation.}

We use standard $\bigO(\cdot)$, $\Omega(\cdot)$ and $\Theta(\cdot)$ order-of-magnitude notation, as well as their tilde variants $\tilde \bigO(\cdot)$, $\tilde\Omega(\cdot)$, $\tilde\Theta(\cdot)$ where lower-order log factors of any variables are suppressed. Note that as we are giving finite sample fully empirical bounds with explicit multiplicative constants, the logarithm base is relevant and we write $\ln \left( \cdot \right)$ for the natural logarithm of base $e$.

\section{Formal statement of the results}
\label{section:formal-results}

Unless otherwise specified, the results in this section are all summarized in terms of learning the parameters of interest up to some $\eps$
relative error and confidence $\delta$. For absolute error point estimation learning bounds, we refer the reader to Section~\ref{section:full-proofs}.

\begin{theorem}[Minimum stationary probability estimation upper bound (relative error)]
\label{theorem:pi-star-ub}
There exists an estimator $\estpimin$ which, for all $0 < \eps < 1$, $0 < \delta < 1 $, satisfies the following. If $\estpimin$ receives as input a sequence $\X=(X_1,\ldots,X_m)$ of length at least $m\ub$ drawn according to an unknown $d$-state Markov chain $(\mc,\bmu)$ with minimal stationary probability $\pimin$ and pseudo-spectral gap $\pssg$, then $\abs{\estpimin - \pimin} < \eps \pimin$ holds with probability at least $1 - \delta$. The sample complexity is upper-bounded by
\begin{equation}
  \label{eq:ubpi}
\begin{split}
  m \ub
  &\eqset \cfrac{C_\star}{\pssg \eps^2 \pimin} \ln{\left( \frac{d}{\delta} \sqrt{\pimin^{-1}} \right)},
\end{split}
\end{equation}
where $C_\star$ is a universal constant.
\end{theorem}

From this theorem, and the fact that for reversible Markov chains the absolute and pseudo-spectral gaps are within a constant multiplicative factor (see Lemma~\ref{lemma:pssg-reversible}), we immediately recover the point estimation upper bound
for the minimum stationary probability provided by \citet{hsu2015mixing} for reversible chains.

\begin{theorem}[Minimum stationary probability estimation lower bound (relative error)]
\label{theorem:pi-star-lb}
Let $d \in \N, d \geq 4$. For every $0 < \eps < 1/2, 0 < \pimin < \frac{1}{d}$ and $
\pimin
<
\pssg
<
1
$, there exists a $(d+1)$-state Markov chain $\mc$ with pseudo-spectral gap $\pssg$ and minimum stationary probability $\pimin$ such that any estimator must require a sequence $\X=(X_1,\ldots,X_m)$ drawn from the unknown $\mc$ of length at least
\begin{equation}
  \label{eq:lbpi}
m\lb
\eqset
\Omega\left( \cfrac{ \ln{\left(\delta^{-1}\right)}}{\pssg  \eps^2 \pimin}\right),
\end{equation}
for $\abs{\estpimin - \pimin} < \eps \pimin$ to hold with probability $\ge1-\delta$.
\end{theorem}

The upper and lower bounds in (\ref{eq:ubpi}) and (\ref{eq:lbpi})
match up to logarithmic factors --- and continue to hold for the reversible case
(Lemma~\ref{lemma:pssg-reversible}).

\begin{theorem}[Pseudo-spectral gap estimation upper bound (absolute error)]
\label{theorem:pseudo-sg-ub}
There exists an estimator $\estpssg$ which, for all $0 < \eps < 1$, $0 < \delta < 1 $, satisfies the following. If $\estpssg$ receives as input a sequence $\X=(X_1,\ldots,X_m)$ of length at least $m \ub$ drawn according to an unknown $d$-state Markov chain $(\mc,\bmu)$ with minimal stationary probability $\pimin$ and pseudo-spectral gap $\pssg$,
then
$\abs{\estpssg - \pssg} < \eps$ holds with probability at least $1 - \delta$.
The sample complexity is upper-bounded by
\begin{equation}
\begin{split}
m \ub \eqset \frac{C_{\mathsf{ps}}}{\pimin \eps^2}  \max \set{ \frac{1}{\pssg},  \mathcal{C}(\mc)}  \ln \left( \frac{d \sqrt{\pimin^{-1}}}{\eps^2 \delta} \right)
\end{split}
,
\end{equation}
where $\mathcal{C}(\mc) \leq \nrm{\mc}_{\bpi}  \min \set{d, \nrm{\mc}_{\bpi}}$,  $\nrm{\mc}_{\bpi} \eqdef \max_{(i,j) \in [d]^2}\set{ \frac{\pi_i}{\pi_j} }$, and $C_{\mathsf{ps}}$ is a universal constant. 
\end{theorem}

The quantity $\mathcal{C}(\mc)$ can be thought of as a measure of ``distance to double stochasticity'' of the chain,
or a measure of non-uniformity of
its stationary distribution;
for doubly-stochastic chains, we have
$\mathcal{C}(\mc) = 1$.
The proof of the theorem provides a more delicate yet less
tractable expression
for $\mathcal{C}(\mc)$,
given in
\eqref{equation:pssg-actual-ub}.

\begin{theorem}[Pseudo-spectral gap estimation lower bound (relative error)]
\label{theorem:pseudo-sg-lb}
Let $d \in \mathbb{N}, d \geq 4$. For every $0 < \eps < 1/4, 0 < \delta < 1/(8d), 0 < \pssg < 1/8$, there exists a $d$-state Markov chains $\mc$ with pseudo-spectral gap $\pssg$ such that every estimator $\estpssg$ must require in the worst case a sequence $\X=(X_1,\ldots,X_m)$ drawn from the unknown $\mc$ of length at least
\beq
m \lb \eqset \Omega \left( \frac{d}{\pssg \eps^2} \ln \left( \frac{1}{d\delta}\right) \right),
\eeq

in order for $\abs{\estpssg - \pssg} < \eps \pssg$ to hold with probability at least $1 - \delta$.

\end{theorem}

The proof of the above result also yields the lower bound of $\tilde{\Omega} \left( \frac{d}{\asg \eps^2} \right)$ is the reversible case, and closes the minimax estimation gap, also in terms of the precision parameter $\eps$, at least for doubly-stochastic Markov chains (for which $\pimin = 1/d$), matching the upper bound of \citet{hsu-mixing-ext} up to logarithmic terms. For a comparison with existing results in the literature,
see
Table~\ref{table:comparison}.

\begin{remark} The minimax sample complexity in
$\pssg$ with relative error remains to be pinned down as our current bounds exhibit a gap. 
In the reversible case, the dependence on $\asg$ was improved from $\asg^3$ in \citet{hsu2015mixing} to
$\asg$ in \citet{levin2016estimating,hsu-mixing-ext}
via a ``doubling trick''
\citep{levin2009markov},
which
exploited
the identity $\asg(\mc^k) = 1 - (1 - \asg(\mc))^k$.
The non-reversible analogue
$\sg ( (\mc \rev )^k \mc^k ) = \sg ( (\mc \rev \mc )^k )$
in general fails,
and we leave the question as an open problem.
\end{remark}

\begin{remark}
We conjecture that the dependence on $\mathcal{C}(\mc)$ is an artifact of the current analysis and that the correct rate is $\tilde \Theta (\tmix / (\pimin \eps^2))$, similar to that of the reversible case. Indeed, suppose we take a reversible chain and lightly perturb some of the entries; a heuristic continuity argument would seem to argue against an abrupt increase in sample complexity.  
\end{remark}

\section{Empirical procedure}

The full procedure for estimating the pseudo-spectral gap is described below. For clarity, the computation of the confidence intervals is not made explicit in the pseudo-code and the reader is referred to Theorem~\ref{theorem:pssg-confidence-intervals-full} for their expression. The estimator is based on a plug-in approach. Namely, we compute an approximate pseudo-spectral gap over a prefix $[K] \subsetneq \N$, of multiplicative reversiblizations of powers of the chain, which are estimated with the natural counts based on observed skipped chains, and averaged over different possible offsets. We additionally introduce a smoothing parameter $\alpha$, which for simplicity is kept identical for each power, and whose purpose is to keep the confidence intervals and estimator properly defined even in degenerate cases. The sub-procedure that computes the spectral gap of the multiplicative reversiblization of a chain uses a few analytical shortcuts to reduce its computational cost, which are developed in depth in Section~\ref{section:reduction-to-radii}.

\begin{algorithm}[H]
 \SetAlgoFuncName{PseudoSpectralGapEstimator}{pssg}
 \SetKwData{CI}{$CI$}
 \SetKwData{N}{$\mathbf{N}$}
 \SetKwData{estimator}{$\estpssg$}
 \SetKwFunction{SpectralGapMultRev}{SpectralGapMultRev}
 \SetKwProg{Fn}{Function}{:}{}
	
	\SetKwFunction{FPSSG}{PseudoSpectralGap}
	\Fn{\FPSSG{$d$, $\alpha$, $(X_1, \dots, X_m)$, $K$}}{
		$\estimator \leftarrow 0$ \\
		\For{$k \leftarrow 1$ \KwTo $K$}{
			$\tilde{\sg}^\dagger_k \leftarrow 0$ \\
			\For{$r \leftarrow 0$ \KwTo $k - 1$}{
					$\hat{\sg}^\dagger_{k,r} \leftarrow  \SpectralGapMultRev(d, \alpha, (X_{r + k}, X_{r + 2k}, \dots, X_{r + \floor{(m-r)/k}k} ))$ \\
					$\tilde{\sg}^\dagger_k \leftarrow \tilde{\sg}^\dagger_k + \hat{\sg}^\dagger_{k,r}$
			}
			$\tilde{\sg}^\dagger_k \leftarrow \tilde{\sg}^\dagger_k / k$ \\
			\If{$\tilde{\sg}^\dagger_k / k > \estimator $}{
								$\estimator \leftarrow \tilde{\sg}^\dagger_k / k$
			}
	 }
	 \KwRet \estimator
	}
	
	\SetKwFunction{FSGMR}{SpectralGapMultRev}
	\Fn{\FSGMR{$d, \alpha, (X_1, \dots, X_n)$}}{
				 \SetKwData{N}{$\mathbf{N}$}
				 \SetKwData{estimator}{$\hat{\sg}^\dagger$}
				 \SetKwData{D}{$\mathbf{D}$}
				 \SetKwData{NT}{$\mathbf{T}$}
				 \SetKwData{NN}{$\mathbf{N}$}
				 \SetKwData{NS}{$\mathbf{S}$}
				 \SetKwData{F}{$\mathbf{F}$}
				 \SetKwData{G}{$\mathbf{G}$}
				 \SetKwFunction{SpectralRadius}{SpectralRadius}
				 \SetKwFunction{MatrixMultiply}{MatrixMultiply}
				 \SetKwFunction{Diag}{Diag}
        $\NN \leftarrow \left[ d \alpha \right]_{d}$ \\
				$\NT \leftarrow \left[ \alpha \right]_{d \times d}$ \\
				\For{$t \leftarrow 1$ \KwTo $n - 1$}{
						$\NN[X_t] \leftarrow \NN[X_t] + 1$ \\
						$\NT[X_t, X_{t+1}] \leftarrow \NT[X_t, X_{t+1}] + 1$ \\
				}
				$\NS \leftarrow \sqrt{1 / \NN}$ \\
				$\F \leftarrow \MatrixMultiply(\NS \trn, \NS) / n$ \\
				$\D \leftarrow \Diag(1 / \NN)$ \\
				$\G \leftarrow \MatrixMultiply(\NS, \NT \trn,  \D, \NT, \NS)$ \\
				$\estimator \leftarrow 1 - \SpectralRadius(\G - \F)$ \\
				\KwRet \estimator
  }

\caption{The estimation procedure outputting $\estpssg$}
\label{algorithm:pseudo-spectral-gap-estimator}
\end{algorithm}

We construct fully empirical confidence intervals whose
non-asymptotic form is deferred to
Theorem~\ref{theorem:pssg-confidence-intervals-full},
and the asymptotic behavior is summarized as follows:

\begin{theorem}[Confidence intervals, asymptotic behavior]
  \label{theorem:pssg-confidence-intervals}
	In the non-reversible case, the interval widths asymptotically behave as
\begin{equation}
\begin{split}
\abs{\estpimin - \pimin} = \tilde \bigO \left(  \frac{\sqrt{d}}{\pssg \sqrt{\pimin m} } \right), \quad  \abs{\estpssgK - \pssg} = \tilde{\bigO} \left( \frac{1}{K} + \sqrt{\frac{d}{\pimin m}}  \left( \sqrt{d} \nrm{\mc}_{\bpi}  + \frac{1}{\pssg \pimin} \right) \right),
\end{split}
\end{equation}
and in the reversible case, they asymptotically behave as
\begin{equation}
\begin{split}
\abs{\estpimin - \pimin} = \tilde{\bigO} \left( \frac{\sqrt{d}}{\asg \sqrt{\pimin m} } \right), \qquad  \abs{\estasg - \asg} = \tilde{\bigO} \left( \sqrt{\frac{d}{\pimin m}} \left( \sqrt{d} \nrm{\mc}_{\bpi} + \frac{1}{\asg \pimin} \right) \right).
\end{split}
\end{equation}
\end{theorem}

See Remark~\ref{remark:improvement-empirical-intervals}
for a discussion of how
Theorem~\ref{theorem:pssg-confidence-intervals}
improves the state of the art.

\section{Proof sketches}

In this section we,
provide proof sketches and explain the basic intuition. These are fully fleshed out in
Section~\ref{section:full-proofs}.

\subsection{Point estimation}
\subsubsection{Proof sketch of Theorem~\ref{theorem:pi-star-ub}}
We take the natural candidate
$\estpimin \eqdef \frac{1}{m}  \min_{i \in [d]} \abs{t \in [m]: X_t = i}$
as our estimator.
The proof follows along the lines of its counterpart in
\citet{hsu2015mixing} for the reversible case, with the exception that it makes use of a more general
concentration inequality from \citet{paulin2015concentration} that also applies to non-reversible Markov chains.

\subsubsection{Proof sketch of Theorem~\ref{theorem:pi-star-lb}}

To prove the claim, we first focus on the estimation problem with absolute error,
in
the regime $2 \eps < \pimin < \pssg$.
We  construct the following star-shaped class of \emph{reversible} Markov chain where
a single ``hub'' state is special,
while the remaining
``spoke''
can only transition to themselves or to the hub.
For $d \in \mathbb{N}, d \geq 4$,
\begin{equation}
\Smc = \set{ \smc_\alpha(\dist): 0 < \alpha < 1, \dist  = (p_1, \dots, p_d) \in \Delta_d } \text{ where } \smc_\alpha(\dist) = \left( \begin{smallmatrix} \alpha & (1 - \alpha) p_1 & \cdots & (1 - \alpha) p_d \\ \alpha & 1 - \alpha & \hdots & 0 \\ \vdots & \vdots & \ddots & \vdots \\ \alpha & 0 & \hdots & 1 - \alpha \end{smallmatrix} \right)
.
\end{equation}
Any chain in $\Smc$ is readily computed to have
stationary distribution $\left(\alpha, (1 - \alpha)p_1, \dots, (1 - \alpha)p_d \right)$ and pseudo-spectral
gap $\pssg(\smc_\alpha(\dist)) = \Theta\left( \alpha \right)$. 
Consider
$\dist,\dist_\eps\in\Delta_d$ defined by
\begin{equation}
\begin{split}
\dist \eqset \left(\beta , \beta, \frac{1 - 2 \beta}{d-2} , \dots, \frac{1 - 2 \beta}{d-2} \right), \dist_\eps \eqset \left(\beta + 2 \eps , \beta - 2 \eps, \frac{1 - 2 \beta}{d-2} , \dots, \frac{1 - 2 \beta}{d-2} \right),
\end{split}
\end{equation}
with $2 \eps < \beta < 1/d$.
Consider further
the two (stationary) Markov chains $\smc_{\alpha},\smc_{\alpha,\eps}\in\Smc$
indexed, respectively, by $\dist$ and $\dist_\eps$.
Then $\abs{\pimin\left( \smc_{\alpha} \right) - \pimin\left( \smc_{\alpha,\eps} \right)} = \Omega(\eps)$, and for all $m$, we
exploit the structure of $\Smc$ to derive
a tensorization property of the KL divergence between trajectories
$\X_1^m \sim \smc_{\alpha}$ and $\Y_1^m \sim \smc_{\alpha, \eps}$:
$$\kl{\X_1^m}{\Y_1^m} \leq \alpha m \kl{\dist_\eps}{\dist}.$$
The argument is concluded with
a direct computation of $\kl{\dist_\eps}{\dist} =\bigO \left( \frac{\eps^2}{\beta} \right)$
and a KL version of Le Cam's two point method.

\subsubsection{Proof sketch of Theorem~\ref{theorem:pseudo-sg-ub}}

We solve the estimation problem with respect to the absolute error. The quantity we wish to estimate is a maximum over the integers,
while an empirical procedure can only consider a finite search space $[K]\subset\N$.
To highlight the dependence of our estimator on the choice of $K$, we write
\begin{equation}
\label{eq:sgk-def}
\begin{split}
\sg^{\dagger}_k \eqdef \sg\left( \left(\mc \rev \right)^k \mc^k\right), \quad \widehat{\sg}^{\dagger}_k \eqdef \sg\left( \left(\widehat{\Lsym}^{(k)}\right) \trn \widehat{\Lsym}^{(k)} \right), \quad \pssgK \eqdef \max_{k \in [K]} \set{ \frac{\sg^{\dagger}_k}{k}},
\end{split}
\end{equation}
where $\widehat{\Lsym}^{(k)}$ the empirical
version of the rescaled transition matrix
\eqref{eq:Ldef}
associated with the
$k$-skipped chain. We denote by $\estpssgK$ the empirical estimator, chosen in our case to be 
\begin{equation}
\begin{split}
\estpssgK \eqdef \max_{k \in [K]}
\set
{
\frac{\widehat{\sg}^{\dagger}_k}{k}
}.
\end{split}
\end{equation}
It is easily seen that
$\abs{\pssgK - \pssg} \leq \frac{1}{K}$,
and so it suffices to consider
$K = \ceil{2/\eps}$,
which  yields
\begin{equation}
\begin{split}
\PR{\abs{\estpssgK - \pssg} > \eps} \leq \sum_{k = 1}^{\ceil{\frac{2}{\eps}}} \PR{ \abs{\sg^{\dagger}_k - \widehat{\sg}^{\dagger}_k}  > \frac{k \eps}{2}}.
\end{split}
\end{equation}
Via a standard application of
Weyl's eigenvalue perturbation inequality, Perron-Frobenius theory, and properties of similar matrices,
we obtain:
\begin{equation}
\begin{split}
\abs{\widehat{\sg}^{\dagger}_k - \sg^{\dagger}_k} \leq \nrm{(\widehat{\Lsym}^{(k)}  ) \trn \widehat{\Lsym}^{(k)} - (\Lsym^k ) \trn \Lsym^k}_2 \leq 2 \nrm{\Lsym^k - \widehat{\Lsym}^{(k)}}_2.
\end{split}
\end{equation}
We continue by decomposing $\widehat{\Lsym}^{(k)} - \Lsym^k$,
\begin{equation}
\begin{split}
&\widehat{\Lsym}^{(k)} - \Lsym^k =\EE{\mc}^{(k)} + \EE{\bpi,1}^{(k)} \Lsym^k  + \Lsym^k \EE{\bpi,2}^{(k)} + \EE{\bpi,1}^{(k)} \Lsym^k \EE{\bpi,2}^{(k)} \\
\text{where } &\EE{\mc}^{(k)} = \left(\estDpi^{(k)}\right)^{1/2} \left(\estmc^{(k)} - \mc^k \right) \left(\estDpi^{(k)}\right)^{-1/2}, \\
&\EE{\bpi,1}^{(k)} = \left(\estDpi^{(k)}\right)^{1/2}\Dpi^{-1/2} - \identity, \qquad \EE{\bpi,2}^{(k)} = \Dpi^{1/2}\left(\estDpi^{(k)}\right)^{-1/2} - \identity \\
\end{split}
\end{equation}
so that
\begin{equation}
\begin{split}
\nrm{\widehat{\Lsym}^{(k)} - \Lsym^k}_2 &\leq \nrm{\EE{\mc}^{(k)}}_2 + 2\nrm{ \EE{\bpi}^{(k)}}_2 + \nrm{\EE{\bpi}^{(k)}}_2^2, \text{ where }\nrm{\EE{\bpi}^{(k)}}_2 \eqdef \max \set{\nrm{\EE{\bpi,1}^{(k)}}_2, \nrm{\EE{\bpi,2}^{(k)}}_2}.
\end{split}
\end{equation}
Now we
must upper bound
$\nrm{\EE{\mc}^{(k)}}_2$,
$\nrm{ \EE{\bpi}^{(k)}}_2$,
and
the ``bad'' event that a $k$-skipped
chain did not visit
every state
a ``reasonable'' amount of times (in which case,
the estimator might not even be properly defined).
The most challenging quantity to bound is $\nrm{\EE{\mc}^{(k)}}_2$, which we achieve via a
row-martingale process.
After carefully controlling its second-order induced row and column processes,
we invoke a matrix martingale version of Freedman's inequality \citep{tropp2011freedman},
concluding that $\nrm{\EE{\mc}^{(k)}}_2 = \bigO(k \eps)$
with high confidence for a trajectory of
length
$m = \tilde \Omega \left( \frac{\nrm{\mc ^k}_1 \nrm{\mc}_{\bpi} }{k \eps^2} \right)$.
The bound $\nrm{ \EE{\bpi}^{(k)}}_2 = \bigO \left( k \eps \right)$
follows from
\citet[Section~6.3]{hsu-mixing-ext}
and
Theorem~\ref{theorem:alternative-concentration-pssg}.

\subsubsection{Proof sketch of Theorem~\ref{theorem:pseudo-sg-lb}}
For $0 < \alpha < \frac{1}{8}$ and $d \geq 4$,
we define the following family of {symmetric} (and hence {reversible}) stochastic $d\times d$ matrices:
\begin{equation}
\begin{split}
\mc(\alpha) = \begin{pmatrix}
1 - \alpha & \frac{\alpha}{d-1} & \cdots & \cdots & \frac{\alpha}{d-1} \\
\frac{\alpha}{d-1} & 1/2 - \frac{\alpha}{d-1} & \frac{1}{2(d-2)} & \cdots & \frac{1}{2(d-2)} \\
\vdots & \frac{1}{2(d-2)} & \ddots &  & \frac{1}{2(d-2)} \\
\vdots & \vdots & & & \vdots \\
\frac{\alpha}{d-1} & \frac{1}{2(d-2)} & \frac{1}{2(d-2)} & \cdots & 1/2 - \frac{\alpha}{d-1} \\
\end{pmatrix},
\end{split}
\end{equation}
for which $\bpi = \frac{1}{d} \cdot \unit$, and $\pssg = \Theta (\alpha)$. We then
invoke
a result of \citet{kazakos1978bhattacharyya}, reproduced
in Lemma~\ref{lemma:mc-hellinger-recursive},
which provides with a
method for recursively computing the Hellinger distance between two
distributions over words of length $m$ sampled from different Markov chains in terms of the
entry-wise geometric mean of their transition matrices.
The problem
is then
reduced to one of controlling a spectral radius.
Writing for convenience 
$$ p=\frac{\sqrt{\alpha_0 \alpha_1}}{d-1}, q=\frac{1}{2(d-2)}, r=\sqrt{(1-\alpha_0)(1-\alpha_1)},  s = \sqrt{ \left( 1/2 - \frac{\alpha_0}{d-1} \right)\left(1/2 - \frac{\alpha_1}{d-1} \right)},$$
we compute the entry-wise geometric mean to be
\begin{equation}
\begin{split}
\left[ \mc(\alpha_0), \mc(\alpha_1) \right]_{\surd} = \begin{pmatrix}
r & p & \cdots & \cdots & p \\
p & s & q & \cdots & q \\
\vdots & q & \ddots &  & q \\
\vdots & \vdots & & & \vdots \\
p & q & q & \cdots & s \\
\end{pmatrix}
\end{split}
.
\end{equation}
Observing
that the rank of this matrix is less than $2$
and employing some careful analysis,
we are able to
bound
the spectral radius $\rho$ of $\left[ \mc(\alpha_0), \mc(\alpha_1) \right]_{\surd}$ from below by
$\rho \geq 1 - 6 \frac{\alpha \eps^2}{d-1}$.
A Hellinger version of Le Cam's two-point method
concludes the proof.

\subsection{Empirical confidence intervals}
\paragraph{Non-reversible setting.}
The complete algorithmic procedure is described in Algorithm~\ref{algorithm:pseudo-spectral-gap-estimator}. Formally, for a sample path $(X_1, \dots, X_m)$, a fixed $k$ and a smoothing parameter $\alpha$, we construct the estimator
\begin{equation}
\begin{split}
\widetilde{\sg}^\dagger_{k, \alpha} (X_1, \dots, X_m) \eqdef \frac{1}{k}\sum_{r = 0}^{k-1} \widehat{\sg}_{k,r, \alpha}^\dagger (X_{t}^{(k,r)}, 1 \leq t \leq \floor{(m-r)/k}),
\end{split}
\end{equation}
where $\widehat{\sg}_{k,r, \alpha}^\dagger$ is an estimator for the spectral gap of the multiplicative reversiblization of $\mc^k$, which was constructed by observing the $k$-skipped $r$-offset Markov chain $X_{t}^{(k,r)}, 1 \leq t \leq \floor{(m-r)/k}$,
and applying Laplace $\alpha$-smoothing. We
then notice that
$\widehat{\sg}_{k,r, \alpha}^\dagger = \sg \left( \left(\widehat{\Lsym}^{(k,r, \alpha)}\right) \trn \widehat{\Lsym}^{(k,r, \alpha)} \right)$, where
\begin{equation}
\begin{split}
\left(\widehat{\Lsym}^{(k,r, \alpha)}\right) \trn \widehat{\Lsym}^{(k,r, \alpha)}  &= \left( \mathbf{D}_N^{(k,r,\alpha)} \right)^{-1/2}  \left(\mathbf{N}^{(k ,r, \alpha)}\right) \trn  \left( \mathbf{D}_N^{(k,r,\alpha)} \right)^{-1} \mathbf{N}^{(k ,r, \alpha)} \left( \mathbf{D}_N^{(k,r,\alpha)} \right)^{-1/2}, \\
\mathbf{N}^{(k,r,\alpha)} &\eqdef \left[N_{i j}^{(k ,r)} + \alpha \right]_{(i,j)}, \; \mathbf{D}_N^{(k,r,\alpha)} \eqdef \diag \left(N_1^{(k,r)} + d \alpha, \dots , N_d^{(k,r)} + d \alpha \right).
\end{split}
\end{equation}
The derivation of the confidence intervals starts with an empirical version of the decomposition introduced for the point estimator. The subsequent analysis has two key components. The first is
a perturbation bound for the stationary distribution in terms of the pseudo-spectral gap
and the stability of the perturbation of matrix with respect to the $\nrm{\cdot}_\infty$ norm.
More precisely,
Lemma~\ref{lemma:non-reversible-kappa-control} guarantees that
\begin{equation}
\begin{split}
\nrm{\estbpi - \bpi}_\infty \leq \tilde{\bigO} \left(1 \right) \frac{1}{\pssg\left(\estmc\right)} \nrm{\estmc - \mc}_\infty.
\end{split}
\end{equation}
The second
component (Lemma~\ref{lemma:mc-empirical-learning-infinity-norm}) involves controlling
the
latter perturbation
in terms of empirically observable quantities.
In particular,
\begin{equation}
\begin{split}
\nrm{\estmc - \mc}_\infty \leq \tilde{\bigO} \left( 1 \right) \sqrt{\frac{d}{\Nmin}} 
\end{split}
\end{equation}
holds with high probability
---
which is an empirical version of the result of \citet[Theorem~1]{wolfer2018minimax},
achieved by constructing and analyzing appropriate row-martingales. 

\paragraph{Reversible setting.}
Our analysis also
yields
improvements
over
the state of the art estimation procedure in the reversible setting,
where
\citet{hsu2015mixing}
used
the absolute spectral gap $\asg$ of the additive reversiblization of the empirical transition matrix
$\frac{\estmc \rev + \estmc}{2}$
as the estimator for the mixing time.
Our analysis via row-martingales
sharpens
the confidence intervals
roughly by a factor of
$\bigO(\sqrt{d})$ over the previous method.
The latter
relied on
entry-wise martingales together with the metric inequality $\nrm{\mathbf{A}}_\infty \leq d \max_{(i,j) \in [d]^2} \abs{\mathbf{A}(i, j)}, \mathbf{A} \in \R^{d \times d} $.
Additionally, we show that the computation complexity of the task can be
reduced
over
non-trivial parameter regimes. We
achieve this via
iterative methods
for
computing the second largest eigenvalue,
and
by replacing an expensive pseudo-inverse computation
by
the
already-computed
estimator for
$\asg$ itself (Corollary~\ref{corollary:reversible-kappa-control}).
These computational improvements
do not degrade
the asymptotic behavior of the confidence intervals.

\section*{Acknowledgments}
We are thankful to Daniel Paulin for the insightful
conversations, and to the anonymous referees for their valuable comments.

\bibliography{bibliography}
\bibliographystyle{abbrvnat}

\appendix
\section{Summary table}

We use the notation $\tilde \bigO_{\times}, \tilde \bigO_{+}, \tilde \Omega_{\times}, \tilde \Omega_{+}$ to denote sample complexity (where logarithmic factors are suppressed) upper and lower bounds for estimation with absolute $(+)$ and relative $(\times)$ error.

\begin{center}
\begin{table}[!ht]
  \begin{tabular}{ l | c | c | c }
    \hline
    Quantity & \citet{hsu2015mixing} & \citet{levin2016estimating} & Present work \\ \hline
		$\pimin$ (rev.) & $\tilde{\bigO}_{\times} \left( \cfrac{1}{\pimin \asg \eps^2} \right)$ & - & $\tilde{\Omega}_{\times}\left( \cfrac{1}{\pimin \asg \eps^2} \right)$ Th.~\ref{theorem:pi-star-lb} \\ \hline
		$\pimin$ (non-rev.) & - & - & $\tilde{\Theta}_{\times}\left( \cfrac{1}{\pimin \pssg \eps^2} \right)$ Th.~\ref{theorem:pi-star-ub}, Th.~\ref{theorem:pi-star-lb} \\ \hline		
    $\asg$ (rev.) & \shortstack{$\tilde{\bigO}_{+} \left( \cfrac{1}{\pimin \asg \eps^2} \right)$ \\ $\tilde{\bigO}_{\times} \left( \cfrac{1}{\pimin \asg^3 \eps^2} \right)$ \\ $\tilde{\Omega}\left( \cfrac{d}{\asg} + \frac{1}{\pimin} \right)$} & $\tilde{\bigO}_{\times} \left( \cfrac{1}{\pimin \asg \eps^2} \right)$ & $ \tilde{\Omega}_{\times} \left( \cfrac{d}{\asg \eps^2} \right)$ Th.~\ref{theorem:pseudo-sg-lb} \\ \hline	
    $\pssg$ (non-rev.) & - & - & \shortstack{$\tilde{\bigO}_{+}\left( \frac{\max \set{ \pssg^{-1},  \mathcal{C}(\mc)}}{\pimin \eps^2} \right)$ Th.~\ref{theorem:pseudo-sg-ub} \\  $ \tilde{\Omega}_{\times}\left( \cfrac{d}{\pssg \eps^2} \right)$  Th.~\ref{theorem:pseudo-sg-lb}}  \\ \hline			
  \end{tabular}
	\caption{Comparison with existing results in the literature.}
	\label{table:comparison}
	\end{table}
\end{center}
\section{Confidence intervals}

The main algorithmic procedure whose description is given at Algorithm~\ref{algorithm:pseudo-spectral-gap-estimator} yields the following fully empirical confidence intervals.

\begin{theorem}
  \label{theorem:pssg-confidence-intervals-full}
  Let $C_\mathcal{K} \le 192$ be a universal constant and define
  $$\tau_{\delta, m} = \inf\set{t>0:\left(1 + \ceil{\ln(2m/t)}_+\right)  (d + 1)  e^{-t} \leq \delta}
  .$$
  Then $\tau_{\delta,m}  = \bigO \left( \ln \left( \frac{d \ln{m}}{\delta} \right) \right)$
  and
  with probability at least
  $1 - \delta$,

\begin{equation}
\begin{split}
& \abs{\estpssgK^{(\alpha)} - \pssg} \leq \frac{1}{K} + 2 \max_{k \in [K]} \set{\frac{1}{k^2} \sum_{r=0}^{k-1} \left( \hat{a}^{(k, r, \alpha)}  + 2 \hat{c}^{(k, r, \alpha)} + \left(\hat{c} ^{(k, r, \alpha)} \right)^2 \right)},  \\
& \text{where } \begin{cases}
    \hat{a}^{(k,r,\alpha)} &= \sqrt{d}  \frac{\Nmax^{(k,r)} + d \alpha}{\Nmin^{(k,r)} + d \alpha} \hat{d}^{(k,r, \alpha)}  \\
    \hat{b}^{(k,r, \alpha)} &= \frac{C_\mathcal{K}}{\pssg\left(\estmc^{(k,r, \alpha)}\right)} \ln \left( 2 \sqrt{\frac{2 (\floor{(m -r) / k} + d^2 \alpha)}{\Nmin^{(k,r)} + d \alpha }} \right) \hat{d}^{(k,r, \alpha)} \\
    \hat{c}^{(k, r, \alpha)} &= \frac{1}{2} \max \bigcup_{i \in [d]} \set{\cfrac{\hat{b}^{(k,r,\alpha)}}{\cfrac{N_i^{(k,r)} + d \alpha}{\floor{(m-r)/k} + d^2 \alpha}   }, \cfrac{\hat{b}^{(k,r,\alpha)}}{\left[\cfrac{N_i^{(k,r)} + d \alpha}{\floor{(m-r)/k} + d^2 \alpha} - \hat{b}^{(k,r,\alpha)}\right]_+}} \\
		\hat{d}^{(k,r, \alpha)} &= 4 \tau_{\delta/d, \floor{(m - r)/ k }} \sqrt{\frac{d }{\Nmin^{(k,r)} + d \alpha}} + \frac{ 2 \alpha d}{\Nmin^{(k,r)} + d \alpha}.
  \end{cases}
\end{split}
\end{equation}
The interval widths asymptotically behave as
\begin{equation}
\begin{split}
\abs{\estpimin^{(\alpha)} - \pimin} &= \tilde \bigO \left(  \frac{\sqrt{d}}{\pssg \sqrt{\pimin m} } \right) \\   \abs{\estpssgK^{(\alpha)} - \pssg} &= \tilde{\bigO} \left( \frac{1}{K} + \sqrt{\frac{d}{\pimin m}}  \left( \sqrt{d} \nrm{\mc}_{\bpi}  + \frac{1}{\pssg \pimin} \right) \right).
\end{split}
\end{equation}
In the reversible case,
\begin{equation}
\begin{split}
\abs{\estasg^{(\alpha)} - \asg} &\leq \hat{a}^{(\alpha)}  + 2 \hat{c}^{(\alpha)} + \left(\hat{c} ^{(\alpha)} \right)^2\qquad  \abs{\estpimin^{(\alpha)} - \estpimin} \leq \hat{b}^{(\alpha)}, \\
\text{where } & \begin{cases}
    \hat{a}^{(\alpha)} &= \sqrt{d} \frac{\Nmax + d \alpha}{\Nmin + d \alpha} \hat{d}^{(\alpha)}  \\
    \hat{b}^{(\alpha)} &= \frac{C_\mathcal{K}}{\estasg^{(\alpha)}} \ln \left( 2 \sqrt{\frac{2 (m + d^2 \alpha)}{\Nmin + d \alpha }} \right) \hat{d}^{(\alpha)} \\
    \hat{c}^{(\alpha)} &= \frac{1}{2} \max \bigcup_{i \in [d]} \set{\cfrac{\hat{b}^{(\alpha)}}{\cfrac{N_i + d \alpha}{m + d^2 \alpha}   }, \cfrac{\hat{b}^{(\alpha)}}{\left[\cfrac{N_i + d \alpha}{m + d^2 \alpha} - \hat{b}^{(\alpha)}\right]_+}} \\
		\hat{d}^{(\alpha)} &= 4 \tau_{\delta/d, \floor{m }} \sqrt{\frac{d }{\Nmin + d \alpha}} + \frac{ 2 \alpha d}{\Nmin + d \alpha}
  \end{cases}.
\end{split}
\end{equation}
The interval widths asymptotically behave as
\begin{equation}
\begin{split}
 \abs{\estpimin^{(\alpha)} - \pimin} = \tilde{\bigO} \left( \frac{\sqrt{d}}{\asg \sqrt{\pimin m } } \right), \qquad  \abs{\estasg^{(\alpha)} - \asg} = \tilde{\bigO} \left( \sqrt{\frac{d}{\pimin m}} \left( \sqrt{d} \nrm{\mc}_{\bpi} + \frac{1}{\asg \pimin} \right) \right).
\end{split}
\end{equation}

\end{theorem}
\section{Proofs}
\label{section:full-proofs}

\subsection{Minimum stationary probability}
\subsubsection{Proof of Theorem~\ref{theorem:pi-star-ub}}

Let $(X_1, \dots, X_m) \sim (\mc, \bmu)$ be a $d$-state Markov chain,
with pseudo-spectral gap $\pssg$ and stationary distribution $\bpi$ minorized by $\pimin$.
Our estimator for $\pimin$ is
defined
as the minimum of the empirical stationary distribution, or more formally,
\begin{equation}
\label{definition:estimator-pi-star-ub}
\estpimin \eqdef \min_{i \in [d]} \estbpi(i)
=
\min_{i\in[d]}
\estbpi(i) \frac{1}{m} \abs{t \in [m]: X_t = i}.
\end{equation}
Without loss of generality, suppose that $\pimin = \pi_1 \leq \pi_2 \leq \dots \leq \pi_d$
(renumber states if needed).
A Bernstein-type inequality
\citep[Theorem~3.4]{paulin2015concentration}, combined with
\citet[Theorem~3.10]{paulin2015concentration},
yields that for all $i \in [d]$ and $t > 0$,

\begin{equation}
\begin{split}
\PR[\mc, \bmu]{|\estpi_i - \pi_i| \geq t} &\leq \sqrt{2 \nrm{\bmu/\bpi}_{2, \bpi}} \expo{- \frac{t^2 \pssg m }{16(1 + 1/(m \pssg))\pi_i(1-\pi_i) + 40t}}.
\end{split}
\end{equation}
Taking
$m > \frac{1}{\pssg}$ and
putting
$$t_m = \cfrac{\log{\left( \frac{d}{\delta} \sqrt{2 \nrm{\bmu/\bpi}_{2, \bpi}} \right)}}{\pssg m},
\qquad
t = \sqrt{32\pi_i t_m} + 40 t_m,$$
yields (via a union bound)
$\PR[\mc, \bmu]{ \nrm{\estbpi - \bpi}_\infty \geq t } \leq \delta$.
We
claim that
\beqn
\label{eq:piincl}
\forall i \in [d], |\estpi_i - \pi_i| < \sqrt{32\pi_i t_m} + 40 t_m
\implies
|\estpi_\star - \pimin| < 8\sqrt{t_{m} \pimin} + 136 t_{m}
.
\eeqn
Indeed,
Let $i_\star$ be such that $\estpi_\star = \estpi_{i_\star}$, and suppose that $\forall i \in [d]: |\estpi_i - \pi_i| < \sqrt{32\pi_i t_m} + 40 t_m$. Since  $\estpi_\star \leq \estpi_1$,
we have
\begin{equation}
\estpi_\star - \pimin \leq \estpi_1 - \pi_1 \leq \sqrt{32\pimin t_m} + 40 t_m \leq \pimin + 48 t_m,
\end{equation}
where the last inequality follows from the AM-GM inequality. Furthermore, 
$a \leq b \sqrt{a} + c \implies a \leq b^2 + b \sqrt{c} + c$
\citep{bousquet2004introduction}
and
\begin{equation}
\pi_{i_\star} \leq \sqrt{32 t_m }\sqrt{\pi_{i_\star}} + (\estpi_{i_\star} + 40 t_{m}).
\end{equation}
Thus,
\begin{equation}
\begin{split}
\pi_{i_\star} &\leq 32 t_{m} + (\estpi_{\star} + 40 t_{m})  + \sqrt{32 t_m }\sqrt{\estpi_{\star} + 40 t_{m}} = \estpi_{\star} + 72 t_{m}   + \sqrt{32 t_{m}  (\estpi_{\star} + 40 t_{m})} \\
&\leq \estpi_{\star} + \sqrt{32 t_{m}  (2\pimin + 88 t_{m})} + 72 t_{m} \leq \estpi_{\star} + \sqrt{32 \cdot 88 t_{m}^2} +\sqrt{64 t_{m} \pimin} + 72 t_{m} \\
&= \estpi_{\star} + 16 \sqrt{11} t_{m} + 8\sqrt{t_{m} \pimin} + 72 t_{m} \leq \estpi_{\star} + 8\sqrt{t_{m} \pimin} + 136 t_{m}m
\end{split}
\end{equation}
and therefore,
\begin{equation}
\pimin - \estpi_\star \leq \pi_{i_\star} - \estpi_{i_\star} \leq 8\sqrt{t_{m} \pimin} + 136 t_{m},
\end{equation}
whence
\begin{equation}
\abs{\estpi_\star - \pimin} \leq 8\sqrt{t_{m} \pimin} + 136 t_{m}.
\end{equation}
A direct computation shows that 
\begin{equation}
m \geq \cfrac{16^2 \pimin \log{\left( \frac{d}{\delta} \sqrt{2 \nrm{\bmu/\bpi}_{2, \bpi}} \right)}}{\pssg \eps^2 } \implies 8 \sqrt{t_{m} \pimin} \leq \frac{\eps}{2} 
\end{equation}
and 
\begin{equation}
m \geq \cfrac{2 \cdot 136  \log{\left( \frac{d}{\delta} \sqrt{2 \nrm{\bmu/\bpi}_{2, \bpi}} \right)}}{\pssg \eps } \implies 136 t_{m} \leq \frac{\eps}{2},
\end{equation}
so that for $m \geq \cfrac{\log{\left( \frac{d}{\delta} \sqrt{2 \nrm{\bmu/\bpi}_{2, \bpi}} \right)}}{\pssg \eps } \max \set{272, \frac{256 \pimin}{\eps}}$,
we have
$\abs{\estpimin - \pimin} < \eps$ with probability at least $1 - \delta$.
The theorem then follows by choosing the precision to be $\eps \pimin$.

\subsubsection{Proof of Theorem~\ref{theorem:pi-star-lb}}

We first prove the claim for
absolute error
in
the regime $2 \eps < \pimin < \pssg$,
and at the end obtain the claimed result via an absolute-to-relative conversion. Consider a $(d+1)$-state star-shaped class of Markov chains with a single ``hub''
connected to
``spoke'' states,
each of which can only
transition to itself and to the hub.  Namely, we construct the family of $(d+1)$-state Markov chains for $d \in \mathbb{N}, d \geq 4$,
\beq
\Smc = \set{ \smc_\alpha(\dist): 0 < \alpha < 1, \dist  = (p_1, \dots, p_d) \in \Delta_d }, \text{ where } \smc_\alpha(\dist) = \left( \begin{smallmatrix} \alpha & (1 - \alpha) p_1 & \cdots & (1 - \alpha) p_d \\ \alpha & 1 - \alpha & \hdots & 0 \\ \vdots & \vdots & \ddots & \vdots \\ \alpha & 0 & \hdots & 1 - \alpha \end{smallmatrix} \right).
\eeq
Notice that $\bpi(\dist) = \left(\alpha, (1 - \alpha)p_1, \dots, (1 - \alpha)p_d \right)$ is a stationary distribution for $\smc_\alpha(\dist)$, that $\smc_\alpha(\dist)$ is reversible and that the spectrum of $\smc_\alpha(\dist)$
consists of
$\lambda_0 = 1$,
$\lambda_\alpha = 1 - \alpha$ (of multiplicity $d-1$),
and $\lambda_{d} = 0$.
Thus,
the absolute spectral gap is $\asg(\smc_\alpha(\dist)) = \alpha$, and the
pseudo-spectral gap $\pssg=O(\alpha)$,
by
Lemma~\ref{lemma:pssg-reversible}.
We apply
Le Cam's two-point method as follows. Take $\dist$ and $\dist_\eps$ in $\Delta_d$ defined as follows,
\begin{equation}
\begin{split}
\dist \eqset \left(\beta , \beta, \frac{1 - 2 \beta}{d-2} , \dots, \frac{1 - 2 \beta}{d-2} \right), \dist_\eps \eqset \left(\beta + 2 \eps , \beta - 2 \eps, \frac{1 - 2 \beta}{d-2} , \dots, \frac{1 - 2 \beta}{d-2} \right),
\end{split}
\end{equation}
with $2 \eps < \beta < 1/d$.
For readability, we will abbreviate
in this section
$$\smc_{\alpha} \eqset \smc_{\alpha}(\dist),  \bpi \eqset \bpi(\dist),  \smc_{\alpha, \eps} \eqset \smc_{\alpha}(\dist_\eps),  \bpi_\eps \eqset \bpi(\dist_\eps),  \PR{\cdot} \eqset \PR[\smc_{\alpha}, \pi]{\cdot} ,  \PR[\eps]{\cdot} \eqset \PR[\smc_{\alpha, \eps}, \pi_\eps]{\cdot}
.$$
Consider the two stationary Markov chains $(\smc_{\alpha}, \pi)$ and $(\smc_{\alpha, \eps}, \pi_\eps)$ for $\alpha > \beta$. First notice that 
$$\abs{\min_{i \in [d+1]} \pi_\eps(i) - \min_{i \in [d+1]} \pi(i)} = 2 \eps.$$ 
We now exhibit a tensorization property of the KL divergence
between trajectories of length $m$ sampled from
the two
chains. Let $\X_1^m \sim (\smc_{\alpha}, \pi)$ and $\Y_1^m \sim (\smc_{\alpha, \eps}, \pi_\eps)$, from the definition of the KL divergence and the Markov property,
\begin{equation}
\begin{split}
&\kl{\Y_1^m}{\X_1^m} = \sum_{(z_1, \dots, z_m) \in [d]^m} \PR[\eps]{\Y_1^m = z_1^m} \ln \left( \frac{\PR[\eps]{\Y_1^m = z_1^m}}{\PR{\X_1^m = z_1^m}} \right) \\
&= \underbrace{\sum_{(z_1, \dots, z_{m-1}) \in [d]^{m-1}} \PR[\eps]{\Y_1^{m-1} = z_1^{m-1}}  \ln \left( \frac{\PR[\eps]{\Y_1^{m-1} = z_1^{m-1}} }{\PR{\X_1^{m-1} = z_1^{m-1}}} \right)}_{\kl{\Y_1^{m-1}}{\X_1^{m-1}}} \underbrace{\sum_{z_m \in [d]} \smc_{\alpha, \eps} (z_{m-1}, z_m)}_{=1} \\
&+ \sum_{z_1^{m-2}  \in [d]^{m-2}} \PR[\eps]{\Y_1^{m-2} = z_1^{m-2}} \sum_{z_{m-1} \in [d]} \smc_{\alpha, \eps} (z_{m-2}, z_{m-1}) \Bigg\{ \\
& \underbrace {\sum_{z_m \in [d]} \smc_{\alpha, \eps} (z_{m-1}, z_m) \ln \left( \frac{\smc_{\alpha, \eps} (z_{m-1}, z_m) }{\smc_{\alpha} (z_{m-1}, z_m)} \right)}_{= \pred{z_{m-1} = 1} \kl{\dist_\eps}{\dist} } \Bigg\}, \\
\end{split}
\end{equation}
and as by structural property of the chains of the class $\forall z_{m-2} \in [d], \smc_{\alpha, \eps} (z_{m-2}, 1) = \alpha$,
\begin{equation}
\begin{split}
\kl{\Y_1^m}{\X_1^m} &= \kl{\Y_1^{m-1}}{\X_1^{m-1}} + \alpha \kl{\dist_\eps}{\dist} \underbrace{ \sum_{z_1^{m-2} \in [d]^{m-2}} \PR[\eps]{\Y_1^{m-2} = z_1^{m-2}}}_{=1}, \\
\end{split}
\end{equation}
so by induction and stationarity,
$\kl{\Y_1^m}{\X_1^m} = m \alpha \kl{\dist_\eps}{\dist}$.
It remains
to compute the KL divergence between the two distributions,
\begin{equation}
\begin{split}
\kl{\dist_\eps}{\dist} &= (\beta + 2\eps) \ln \left( 1  + \frac{2\eps}{\beta} \right) + (\beta - 2\eps) \ln \left( 1 + \frac{- 2\eps}{\beta} \right) \leq \frac{8 \eps^2}{\beta}.
\end{split}
\end{equation}
Denote by $\M_{d, \pssg, \pimin}$ the collection of all $d$-state Markov chains whose stationary distribution is minorized by $\pimin$ and whose pseudo-spectral gap is at least $\pssg$, and define the minimax risk as
\begin{equation}
\begin{split}
\mathcal{R}_m = \inf_{\estpimin} \sup_{\mc \in \M_{d, \pssg, \pimin}} \PR[\mc]{\abs{\estpimin - \pimin} > \eps},
\end{split}
\end{equation}
then from the KL divergence version of Le Cam's theorem \citep[Chapter~2]{tsybakov2009introduction},
\begin{equation}
\begin{split}
\mathcal{R}_m &\geq  \frac{1}{4} \expo{ - \kl{\Y_1^m}{\X_1^m}} \geq  \frac{1}{4} \expo{ - \frac{8 \alpha \eps^2 m}{\beta} } \geq  \frac{1}{4} \expo{ - \frac{8 \pssg \eps^2 m}{\pimin} }.
\end{split}
\end{equation}
Hence,
for $m \leq \cfrac{\pimin \ln{\left(\frac{1}{4 \delta}\right)}}{8 \pssg \eps^2}$,
we have
$\mathcal{R}_m \geq \delta$,
and
so
$m = \Omega\left( \cfrac{\pimin \ln{\left(\frac{1}{\delta}\right)}}{\pssg \eps^2}\right) = \tilde{\Omega}\left( \cfrac{\pimin}{\pssg \eps^2}\right)$ is a lower bound for the problem, in the $\pssg > \pimin$ regime, up to absolute error. When taking the accuracy to be $\eps \pimin$ instead, the previous bound becomes $\tilde{\Omega}\left( \frac{1}{\pssg \pimin \eps^2}\right)$, and the fact that the proof exclusively makes use of a reversible family confirms that the upper bound derived in \citet{hsu-mixing-ext} is minimax optimal up to a logarithmic a factor,
in the parameters $\pimin, \pssg$ and
$\eps$.

\subsection{Pseudo-spectral gap}
\subsubsection{Proof of Theorem~\ref{theorem:pseudo-sg-ub}}
\label{section:pseudo-sg-ub}
In this section, we analyze
our
point estimator for the pseudo-spectral gap.

\paragraph{Reduction to a maximum over a finite number of estimators.} \label{paragraph:reduction-finite-K}
Recall the definitions from \eqref{eq:sgk-def}:
\beq
\sg^{\dagger}_k \eqdef \sg\left( \left(\mc \rev \right)^k \mc^k\right), \qquad \pssgK \eqdef \max_{k \in [K]} \set{ \frac{\sg^{\dagger}_k}{k} }.
\eeq
It follows from the definition of $\mc \rev$ that
$\bpi$
is
the stationary distribution of $\left(\mc \rev\right)^k \mc^k$
for all
$k \in \N$.
We denote by $\estpssgK$ the empirical estimator for $\pssgK$:
\begin{equation}
\begin{split}
\estpssgK \eqdef \max_{k \in [K]} \set{ \frac{\widehat{\sg}^{\dagger}_k (X_1, \dots, X_m)}{k} },
\end{split}
\end{equation}
where $\widehat{\sg}^{\dagger}_k (X_1, \dots, X_m)$ is an estimator for $\sg^{\dagger}_k$
to be defined below. From the triangle inequality,
\begin{equation}
\begin{split}
\PR{\abs{\estpssgK - \pssg} > \eps} \leq \PR{\abs{\estpssgK - \pssgK} + \abs{\pssgK - \pssg} > \eps}. \\
\end{split}
\end{equation}
By taking a maximum over a larger set, $\abs{\pssgK - \pssg} = \pssg - \pssgK \leq \max_{k \in \N \setminus [K]} \set{ \frac{\sg^{\dagger}_k}{k} }$,
and since $\sg^{\dagger}_k \leq 1$
for all $k \in \N$,
we have
$\abs{\pssgK - \pssg} \leq \frac{1}{K}$.
Thus, for $K \geq \frac{2}{\eps}$,
\begin{equation}
\begin{split}
\PR{\abs{\estpssgK - \pssg} > \eps} \leq \PR{\abs{\estpssgK - \pssgK} + \frac{1}{K} > \eps} \leq  \PR{\abs{\estpssgK - \pssgK}  > \frac{\eps}{2}}, \\
\end{split}
\end{equation}
and from another application of the triangle inequality, for $K \geq \frac{2}{\eps}$,
\begin{equation}
\begin{split}
\PR{\abs{\estpssgK - \pssg} > \eps} \leq \PR{ \max_{k \in [K]} \set{\frac{\abs{\sg^{\dagger}_k - \widehat{\sg}^{\dagger}_k}}{k} } > \frac{\eps}{2}} \leq \sum_{k = 1}^{\ceil{\frac{2}{\eps}}} \PR{ \abs{\sg^{\dagger}_k - \widehat{\sg}^{\dagger}_k}  > \frac{k \eps}{2}}.
\end{split}
\end{equation}

\paragraph{Reduction to controlling spectral norms.}
\label{paragraph:general-k-spectral-reduction}

Recall that $\sg^{\dagger}_k = \sg ((\mc^k) \rev \mc^k )$ is the spectral gap of the multiplicative reversiblization of $\mc^k$, the $k$-skipped Markov chain associated with $\mc$. We now introduce natural estimators for $\mc^k$ and $\bpi$,
\begin{equation}
\begin{split}
\estmc^{(k)}(i,j) \eqdef \frac{N_{i j}^{(k)}}{N_i^{(k)}} &\text{ and }\estbpi^{(k)}(i) \eqdef \frac{N_i^{(k)}}{m}, \\
\end{split}
\end{equation}
where $N_{i}^{(k)}$ and $N_{i j}^{(k)}$ are defined in
(\ref{equation:def-number-of-visits}, \ref{equation:def-number-of-transitions}).
It is readily verified
that $\estbpi^{(k)}$ is
the stationary distribution of $\estmc^{(k)}$,
and so
$(\estmc^{(k)}) \rev \estmc^{(k)}$
is a natural estimator for $(\mc^k) \rev \mc^k$.
We also introduce $\widehat{\Lsym}^{(k)} = (\estDpi^{(k)})^{1/2} \estmc^{(k)} (\estDpi^{(k)})^{-1/2}$, and will later show that the event
$\min_{(k,i) \in [K] \times [d]}N_i^{(k)} > 0$
occurs with high probability for our sampling regime,
in which case
the above quantities
will be well-defined, making smoothing unnecessary.
Notice $(\mc^k) \rev \mc^k = \Dpi^{-1/2} (\Lsym ^k ) \trn \Lsym ^k \Dpi^{1/2}$,
and
$(\Lsym ^k ) \trn \Lsym^k$
is symmetric,
which makes
Weyls' inequality
applicable:
$$\abs{\widehat{\sg}^{\dagger}_k - \sg^{\dagger}_k} \leq \nrm{(\widehat{\Lsym}^{(k)}  ) \trn \widehat{\Lsym}^{(k)} - (\Lsym^k ) \trn \Lsym^k}_2.$$
By the triangle inequality and sub-multiplicativity of the spectral norm, 
\begin{equation}
\begin{split}
\abs{\widehat{\sg}^{\dagger}_k - \sg^{\dagger}_k} \leq \nrm{\left(\Lsym^k \right) \trn}_2  \nrm{\widehat{\Lsym}^{(k)} - \Lsym^k}_2 + \nrm{ \left(\widehat{\Lsym}^{(k)}\right) \trn - \left(\Lsym^k \right) \trn}_2 \nrm{\widehat{\Lsym}^{(k)}}_2 \leq 2 \nrm{\Lsym^k - \widehat{\Lsym}^{(k)}}_2,
\end{split}
\end{equation}
where for the second inequality we
invoked
the Perron-Frobenius theorem ($\sqrt{\bpi}$ is an eigenvector associated to
eigenvalue $1$ for $(\Lsym ^k ) \trn \Lsym^k$,
$k \in [K]$ and $\nrm{\Lsym^k}_2 = 1$;
the same holds
for its empirical version).
We continue by decomposing $\widehat{\Lsym}^{(k)} - \Lsym^k$ into more manageable quantities,
\begin{equation}
\begin{split}
&\widehat{\Lsym}^{(k)} - \Lsym^k =\EE{\mc}^{(k)} + \EE{\bpi,1}^{(k)} \Lsym^k  + \Lsym^k \EE{\bpi,2}^{(k)} + \EE{\bpi,1}^{(k)} \Lsym^k \EE{\bpi,2}^{(k)} \\
\text{where } &\EE{\mc}^{(k)} = \left(\estDpi^{(k)}\right)^{1/2} \left(\estmc^{(k)} - \mc^k \right) \left(\estDpi^{(k)}\right)^{-1/2}, \\
&\EE{\bpi,1}^{(k)} = \left(\estDpi^{(k)}\right)^{1/2}\Dpi^{-1/2} - \identity, \qquad \EE{\bpi,2}^{(k)} = \Dpi^{1/2}\left(\estDpi^{(k)}\right)^{-1/2} - \identity.
\end{split}
\end{equation}
Writing $\nrm{\EE{\bpi}^{(k)}}_2 = \max \set{\nrm{\EE{\bpi,1}^{(k)}}_2, \nrm{\EE{\bpi,2}^{(k)}}_2}$, from the sub-multiplicativity of the spectral norm and another application of the Perron-Frobenius theorem,
\begin{equation}
\begin{split}
\nrm{\widehat{\Lsym}^{(k)} - \Lsym^k}_2 &\leq \nrm{\EE{\mc}^{(k)}}_2 + 2\nrm{ \EE{\bpi}^{(k)}}_2 + \nrm{\EE{\bpi}^{(k)}}_2^2.
\end{split}
\end{equation}
Consider the event
$$\mathcal{K}_{1/2}^K = \set{ \max_{(k,i)\in[K]\times[d]}
\abs{N_i^{(k)} - \E[\mc, \bpi]{N_i^{(k)}}} \leq \frac{1}{2} \E[\mc, \bpi]{N_i^{(k)}}}.$$
Applying the law of total probability,
\begin{equation}
\begin{split}
\PR[\mc, \bmu]{\abs{\estpssgK - \pssg} > \eps} & \leq \PR[\mc, \bmu]{\compt{\mathcal{K}_{1/2}^K}} + \sum_{k = 1}^{\ceil{\frac{2}{\eps}}} \Bigg( \PR[\mc, \bmu]{ \nrm{\EE{\mc}^{(k)}}_2 > \frac{\eps k}{8} \text{ and } \mathcal{K}_{1/2}^K } \\
& + \PR[\mc, \bmu]{ 2\nrm{\EE{\bpi}^{(k)}}_2 + \nrm{\EE{\bpi}^{(k)}}_2^2 > \frac{\eps k }{8} } \Bigg).
\end{split}
\end{equation}

\paragraph{A matrix martingale approach.}

Consider for $k \in [K]$, the following sequence of random matrices $\Y_1^{(k)} = \zero, \Y_t^{(k)} = \left[ \pred{X_{t-1}^{(k)} = i} \left(\pred{X_{t}^{(k)} = j} - \mc^k(i,j)\right) \right]_{(i,j) \in [d]^2}$. Notice that $\sum_{t = 1}^{\floor{m/k}} \Y_t^{(k)} = \diag \left(N_1^{(k)}, \dots, N_d^{(k)}\right) \left(\estmc^{(k)} - \mc^k\right)$. From the Markov property, for any $t \geq 1$, $\E[t-1]{\Y_t^{(k)}} = \zero$, proving that $\Y_t^{(k)}$ is a matrix martingale difference sequence, and we can write 
\begin{equation}
\begin{split}
\EE{\mc}^{(k)} = \diag\left(\sqrt{N_1^{(k)}}, \dots, \sqrt{N_d^{(k)}}\right)^{-1} \left( \sum_{t=1}^{\floor{m/k}} \Y_t^{(k)} \right) \diag\left(\sqrt{N_1^{(k)}}, \dots, \sqrt{N_d^{(k)}}\right)^{-1}.
\end{split}
\end{equation}
At this point, we will
invoke Freedman's inequality for matrices \citep[Corollary~1.3]{tropp2011freedman},
stated here as Theorem~\ref{theorem:matrix-freedman}.
A direct computation shows that for $t \geq 2$,
\begin{equation}
\begin{split}
\nrm{\Y_t^{(k)}}_\infty = 2\left(1 - \mc^k\left(X_{t-1}^{(k)},X_{t}^{(k)}\right)\right) \text{ and } \nrm{\Y_t^{(k)}}_1 = \max_{j \in [d]} \abs{ \pred{X_{t}^{(k)} = j} - \mc^k\left(X_{t-1}^{(k)},j\right) }\\
\end{split}
\end{equation}
so that from H\"older's inequality for matrix norms, $\nrm{\Y_t^{(k)}}_2 \leq \sqrt{\nrm{\Y_t^{(k)}}_1 \nrm{\Y_t^{(k)}}_\infty} \leq \sqrt{2}$. The same trivially holds for $t = 1$.
Similarly, using the Kronecker symbol $\delta_{i j}$, we compute
\begin{equation}
\label{eq:YY}
\begin{split}
\left[\Y_t^{(k)} \left(\Y_t^{(k)}\right) \trn \right]_{(i,j)} &= \delta_{i j} \pred{X_{t-1}^{(k)} = i} \left( 1 - 2 \sum_{\ell = 1}^{d} \mc^k(i,\ell) \pred{X_{t}^{(k)} = \ell}  + \nrm{\mc^k(i,\cdot)}_2^2 \right) \\
\left[ \left(\Y_t^{(k)}\right) \trn \Y_t^{(k)} \right]_{(i,j)} &= \sum_{\ell = 1}^{d} \pred{X_{t-1}^{(k)} = \ell} \bigg(\pred{X_{t}^{(k)} = i} \pred{X_{t}^{(k)} = j} \\
& - \mc^k(\ell,i) \pred{X_{t}^{(k)} = j} - \mc^k(\ell,j) \pred{X_{t}^{(k)} = i} + \mc^k(\ell,i) \mc^k(\ell,j)\bigg).
\end{split}
\end{equation}
Recall the random variables
$\Nmin^{(k)}$ and $\Nmax^{(k)}$ defined in \eqref{equation:def-random-variable-convenient-notation}.
As a consequence of \eqref{eq:YY},
the two predictable quadratic variation processes
$\W_{\col, \floor{m/k}}^{(k)}$ and $\W_{\row, \floor{m/k}}^{(k)}$
are well-defined:
\begin{equation}
\begin{split}
\left[\W_{\col, \floor{m/k}}^{(k)}\right]_{(i,j)} &\eqdef \sum_{t=1}^{\floor{m/k}} \E[t-1]{\left[\Y_t^{(k)} \left(\Y_t^{(k)}\right) \trn \right]_{(i,j)}} = \delta_{i j} N_i^{(k)} \left( 1 - \nrm{\mc^k(i,\cdot)}_2^2 \right) \\
\left[\W_{\row, \floor{m/k}}^{(k)}(i,j)\right]_{(i,j)} &\eqdef \sum_{t=1}^{\floor{m/k}} \E[t-1]{ \left[ \left(\Y_t^{(k)}\right) \trn \Y_t^{(k)} \right]_{(i,j)}} = \sum_{\ell = 1}^{d} N_\ell^{(k)} \mc^k(\ell,i) (\delta_{i j}  - \mc^k(\ell,j))
\end{split}
\end{equation}
whence
\begin{equation}
\begin{split}
\nrm{\W_{\col, \floor{m/k}}^{(k)}}_2 \leq \Nmax^{(k)}, \qquad \nrm{\W_{\row, \floor{m/k}}^{(k)}}_2 \leq \frac{1}{4} \Nmax^{(k)} \nrm{\mc^k}_{1},
\end{split}
\end{equation}
where we used the fact that
$\mathbf{A} \geq \mathbf{B} \geq \zero \text{ (entry-wise) }
\implies \nrm{\mathbf{A}}_2 \geq \nrm{\mathbf{B}}_2$
holds for all real valued matrices.

Finally, as in the event where $\forall k \in [K], \Nmin^{(k)} > 0$, $\EE{\mc}^{(k)}$ is defined, and by sub-multiplicativity of the spectral norm and the fact that for $\mathbf{D} = \diag\left(v_1, \dots, v_d\right)$ a diagonal matrix, $\nrm{\mathbf{D}}_2 = \max_{i \in [d]} v_i$,
\begin{equation}
\begin{split}
\nrm{\EE{\mc}^{(k)}}_2 \leq \frac{1}{\Nmin^{(k)}}\nrm{\sum_{t=1}^{\floor{m/k}} \Y_t^{(k)} }_2. \\
\end{split}
\end{equation}
Since $\mathcal{K}_{1/2}^K \implies
(\Nmin^{(k)} \geq \frac{1}{2} \floor{m/k} \pimin)
\wedge
(\Nmax^{(k)} \leq \frac{3}{2} \floor{m/k} \max_{i \in [d]} \set{ \pi_i })$, writing
\begin{equation}
\nrm{\Sigma_{\floor{m/k}}}_2 \eqdef \max \set{ \nrm{\W_{\col, \floor{m/k}}^{(k)}}_2 , \nrm{\W_{\row, \floor{m/k}}^{(k)}}_2 },
\end{equation}
we have that
\begin{equation}
\begin{split}
& \PR[\mc, \bmu]{ \nrm{\EE{\mc}^{(k)}}_2 > \frac{k \eps}{8}  \text{ and } \mathcal{K}_{1/2}^K } \\
& \leq \PR[\mc, \bmu]{ \nrm{\sum_{t=1}^{\floor{m/k}} \Y_t^{(k)} }_2 > \frac{k \eps}{8} \frac{\floor{m/k} \pimin}{2} \text{ and } \Nmax^{(k)} \leq \frac{3}{2} \floor{m/k} \max_{i \in [d]} \pi_i } \\
& \leq \PR[\mc,\bmu]{\nrm{\sum_{t=1}^{\floor{m/k}} \Y_t }_2 > \frac{(m-k) \eps \pimin}{16}  \text{ and } \nrm{\Sigma_{\floor{m/k}}}_2 \leq \frac{3m}{2k} \nrm{\mc^k}_{1} \max_{i \in [d]} \pi_i } \\
& \leq 2d \expo{- C \cfrac{m \eps^2 \pimin k}{\nrm{\mc}_{\bpi}\nrm{\mc^k}_1 }}
\text{ (Theorem~\ref{theorem:matrix-freedman})},
\end{split}
\end{equation}
and for $m \geq C \frac{ \nrm{\mc^k}_1 \nrm{\mc}_{\bpi}}{ k \eps^2 \pimin} \ln \left( \frac{2 d}{\delta_{\mc}^{(k)}} \right)$, the error probability is controlled by $\delta_{\mc}^{(k)}$. \\\\

\paragraph{Finishing up.}
Since we have
\begin{equation}
\begin{split}
\PR[\mc, \bmu]{ 2\nrm{\EE{\bpi}^{(k)}}_2 + \nrm{\EE{\bpi}^{(k)}}_2^2 > \frac{k \eps }{8}} \leq \PR[\mc, \bmu]{ \nrm{\EE{\bpi}^{(k)}}_2 > \frac{ k\eps }{32} } + \PR[\mc, \bmu]{ \nrm{\EE{\bpi}^{(k)}}_2 > \frac{\sqrt{ k \eps} }{4}},
\end{split}
\end{equation}
Theorem~\ref{theorem:pi-star-ub} together with \citet[Section~6.3]{hsu-mixing-ext}
imply that
for
$m \geq \frac{C}{\pssg \pimin \eps^2} \ln{\left( \frac{d}{\delta_{\bpi}^{(k)}} \sqrt{ \pimin^{-1}} \right)}$,
this quantity is upper bounded by $\delta_{\bpi}^{(k)}$.
Similarly, for $K \geq \frac{2}{\eps}$, taking $m \geq \frac{C_\mathcal{K}}{\pimin \pssg \eps^2} \ln\left( \frac{2\sqrt{2 \pimin^{-1}} d }{ \eps \delta_{\mathcal{K}}} \right)$ and invoking Lemma~\ref{lemma:control-visits-all-k-skipped-chains} gives $\PR[\mc, \bmu]{\compt{\mathcal{K}}} \leq \delta_{\mathcal{K}}$. \\\\
Finally, choosing $\delta_{\mathcal{K}} \eqset \frac{\delta}{4}$, $\delta_{\mc}^{(k)} \eqset \delta_{\bpi}^{(k)} \eqset \frac{\eps \delta}{8}$, and taking the maximum of the three sample sizes,
\begin{equation}
\begin{split}
\label{equation:pssg-actual-ub}
m \geq \frac{C_{\mathsf{ps}}}{\pimin \eps^2} \max \set{ \frac{1}{\pssg},  \nrm{\mc}_{\bpi} \max_{k \in \ceil{2/\eps}} \set{\frac{\nrm{\mc^k}_1}{k} }} \ln \left( \frac{d \sqrt{\pimin^{-1}}}{\eps^2 \delta} \right), C_{\mathsf{ps}} \in \R^+
\end{split}
\end{equation}
is sufficient to control the error up to absolute error $\eps$ and confidence $1 - \delta$.

\begin{remark}
Notice that for all $k \in \N$, 
$$\nrm{\mc^k}_1 = \max_{i \in [d]} \set{\sum_{\ell=1}^d\mc^k(\ell, i)} \leq \frac{1}{\pimin}\max_{i \in [d]} \set{\sum_{\ell=1}^d \pi_\ell\mc^k(\ell, i)} = \nrm{\mc}_{\bpi},$$ 
and
$\nrm{\mc^k}_1 \leq d$, so that
$\frac{\nrm{\mc^k}_1}{k} \leq \min \set{d, \nrm{\mc}_{\bpi}}$;
this justifies the appearance of
$\mathcal{C}(\mc) = \nrm{\mc}_{\bpi} \min \set{d, \nrm{\mc}_{\bpi}}$
in the upper bound.
\end{remark}

\subsubsection{Proof of Theorem~\ref{theorem:pseudo-sg-lb}}

\citet{kazakos1978bhattacharyya}
proposed an efficient method
for recursively computing the Hellinger distance $H^2(\X, \Y)$
between two trajectories $\X = (X_1, \dots, X_m), \Y = (Y_1, \dots, Y_m)$ of length $m$ sampled respectively from two
Markov chains $(\mc_0, \bmu_0)$ and $(\mc_1, \bmu_1)$ in terms of
the entry-wise geometric mean of their transition matrices and initial distributions, 
which we reproduce in Lemma~\ref{lemma:mc-hellinger-recursive}.
\citet[Proof~of~Claim~2]{daskalakis2017testing}
used this result to
upper bound the Hellinger distance in terms of the spectral radius $\rho$
for symmetric stationary Markov chains:
\begin{equation}
\begin{split}
1 - H^2(\X, \Y) &\geq \frac{\rho^m}{d}. \\
\end{split}
\end{equation}
For $0 < \alpha < \frac{1}{8}$, consider the following family of {symmetric} stochastic matrices of size $d \geq 4$:
\begin{equation}
\begin{split}
\mc(\alpha) = \begin{pmatrix}
1 - \alpha & \frac{\alpha}{d-1} & \cdots & \cdots & \frac{\alpha}{d-1} \\
\frac{\alpha}{d-1} & 1/2 - \frac{\alpha}{d-1} & \frac{1}{2(d-2)} & \cdots & \frac{1}{2(d-2)} \\
\vdots & \frac{1}{2(d-2)} & \ddots &  & \frac{1}{2(d-2)} \\
\vdots & \vdots & & & \vdots \\
\frac{\alpha}{d-1} & \frac{1}{2(d-2)} & \frac{1}{2(d-2)} & \cdots & 1/2 - \frac{\alpha}{d-1}
\end{pmatrix}
\end{split}.
\end{equation}
Being {doubly-stochastic}, all the chains described by this family are {symmetric}, {reversible}, and have $\bpi = (1/d , \dots, 1/d)$ as their {stationary distribution}. The eigenvalues of $\mc(\alpha)$ are given by $\lambda_1 = 1, \lambda_{\alpha, 1} = 1 - \frac{d}{d-1} \alpha, \lambda_{\alpha, 2} = 1 - \left( \frac{d-1}{2(d-2)} + \frac{\alpha}{d-1} \right)$. Note that $\lambda_{\alpha,1} > \lambda_{\alpha, 2}$ whenever $\alpha < \frac{d-1}{2(d-2)}$,
and so
$\alpha < \frac{1}{4}
\implies
\sg(\mc(\alpha)) = \frac{d}{d-1} \alpha$. Notice that although the constructed chains are not {lazy},
we still have
$\lambda_{\alpha, 1} > 0$ and $\lambda_{\alpha, 2} > 0$,
so that $\asg(\mc(\alpha)) = \sg(\mc(\alpha))$ and is (by Lemma~\ref{lemma:pssg-reversible}) within a
factor of
$2$
of $\pssg(\mc(\alpha))$.
Let $\bu=(1,0,\ldots,0) \trn $ and $ \bv =\frac1{\sqrt{d-1}}(0,1,\ldots,1) \trn$
and
put
$$ p=\frac{\sqrt{\alpha_0 \alpha_1}}{d-1}, q=\frac{1}{2(d-2)}, r=\sqrt{(1-\alpha_0)(1-\alpha_1)},  s = \sqrt{ \left( 1/2 - \frac{\alpha_0}{d-1} \right)\left(1/2 - \frac{\alpha_1}{d-1} \right)}.$$ 
We proceed to compute
\begin{equation}
\begin{split}
\left[ \mc(\alpha_0), \mc(\alpha_1) \right]_{\surd} &= \begin{pmatrix}
r & p & \cdots & \cdots & p \\
p & s & q & \cdots & q \\
\vdots & q & \ddots &  & q \\
\vdots & \vdots & & & \vdots \\
p & q & q & \cdots & s \\
\end{pmatrix} 
\\ &= (r - s + q )\bu \bu \trn + p \sqrt{d-1} (\bu \bv \trn + \bv \bu \trn) + (d-1) q \bv \bv \trn +(s-q)\identity
.
\end{split}
\end{equation}
So, $\left[ \mc(\alpha_0), \mc(\alpha_1) \right]_{\surd} -(s-q)\identity$ has rank $\leq 2$ and its
operator
restriction to
the subspace
$\operatorname{span}\set{\bu,\bv }$ has a matrix representation $\boldsymbol{R} = \begin{pmatrix}r - s+q & p \sqrt{d-1}  \\ p \sqrt{d-1}  & (d-1)q\end{pmatrix}$. For $d \geq 4$, and since $r - s + q >0$, $\boldsymbol{R}$ is entry-wise positive. Hence
\begin{equation}
\begin{split}
\rho \eqset \rho \left( \left[ \mc(\alpha_0), \mc(\alpha_1) \right]_{\surd} \right) &=s-q+\rho(\boldsymbol{R}),
\end{split}
\end{equation}
with 
\begin{equation}
\begin{split}
\rho(\boldsymbol{R}) &= \max \set{ \frac{\Tr(\boldsymbol{R})}{2} + \sqrt{\frac{\Tr^2(\boldsymbol{R})}{4} - \abs{\boldsymbol{R}}}, \frac{\Tr(\boldsymbol{R})}{2} - \sqrt{\frac{\Tr^2(\boldsymbol{R})}{4} - \abs{\boldsymbol{R}}} }, \\
\end{split}
\end{equation}
so that 
\begin{equation}
\label{equation:rho-definition}
\begin{split}
\rho &=  s-q+\frac{(r-s+d q)+\sqrt{\left[r-s-(d-2)q\right]^2+4(d-1)p^2}}{2} \\ &=\frac{(r+s+1/2)+\sqrt{(r-s-1/2)^2+\frac{4\alpha_0 \alpha_1}{d-1}}}{2}.\\
\end{split}
\end{equation}
Lemma~\ref{lemma:control-rho}
shows that, for $d \geq 4, 0 < \alpha < 1/8, 0 < \eps < 1/2$,
and
$\alpha_0 = \alpha (1 - \eps), \alpha_1 = \alpha (1 + \eps)$,
we have
\beq
\rho = \rho \left( \left[ \mc(\alpha_0), \mc(\alpha_1) \right]_{\surd} \right) \ge
1 - 6 \frac{\alpha \eps^2}{d-1}.
\eeq
The minimax risk for the problem of estimating the pseudo-spectral gap
is defined as
\begin{equation}
\begin{split}
\mathcal{R}_m^{\mathsf{ps}} \eqdef \inf_{\estpssg} \sup_{(\bmu, \mc) \in \M_{d, \pssg, \pimin}} \PR[\mc, \bmu]{\abs{\estpssg(X_1, \dots, X_m) - \pssg(\mc)} > \eps},
\end{split}
\end{equation}
where the $\inf$ is taken over all measurable functions $\estpssg: (X_1, \dots, X_m) \to (0, 1)$,
and the $\sup$ over the set $\M_{d, \pssg, \pimin}$ of $d$-state Markov chains whose minimum stationary probability is $\pimin$,
and of pseudo-spectral gap at least $\pssg$. Using Le Cam's two point method \citep[Chapter~2]{tsybakov2009introduction},
\begin{equation}
\begin{split}
\mathcal{R}_m^{\mathsf{ps}} &\geq \frac{1}{2} \left( 1 - \sqrt{H^2 \left( \X \sim \mc(\alpha_0), \Y \sim \mc(\alpha_1) \right)} \right) \geq \frac{1}{2} \left( 1 - \sqrt{1 - \frac{\rho^m}{d}} \right) \geq \frac{\rho^m}{4 d} \\
&\geq \frac{\expo{m \ln \left( 1 - 6 \frac{\eps^2 \alpha}{d} \right)}}{4d} \geq \frac{\expo{ - 9  \frac{ m\eps^2 \alpha}{d}}}{4d},
\end{split}
\end{equation}
where the last inequality
holds because
$6 \frac{\eps^2 \alpha}{d} \leq 3/8$ and
$\ln(1 - t) \geq -\frac{3}{2} t$,
$t \in \left(0, \frac12 \right)$.
Thus, for $\delta < \frac{1}{4d}$,
a sample of size
at least $m = \Omega \left( \frac{d}{9 \cdot \eps^2 \alpha} \ln \left( \frac{1}{4d\delta}\right) \right)$
is necessary to achieve a
confidence of $1-\delta$.

\begin{lemma}
\label{lemma:control-secondary-square-root}
For $d \geq 4$, $0 < \alpha_0, \alpha_1 < 1/4$, and $s = \sqrt{ \left( 1/2 - \frac{\alpha_0}{d-1} \right)\left(1/2 - \frac{\alpha_1}{d-1} \right)}$, we have
\begin{equation}
\begin{split}
\frac{1}{2} \left( 1 - \frac{\alpha_0 + \alpha_1}{d-1} - 2 \left(\frac{\alpha_0 - \alpha_1}{d-1}\right)^2 \right) \leq s \leq \frac{1}{2} \left( 1 - \frac{\alpha_0 + \alpha_1}{d-1} \right).
\end{split}
\end{equation}
\begin{proof}
We begin by showing that $s \leq \frac{1}{2} \left( 1 - \frac{\alpha_0 + \alpha_1}{d-1} \right)$. By the AM-GM inequality,
$4 \alpha_0 \alpha_1 \leq (\alpha_0 + \alpha_1)^2$, whence
\begin{equation}
\begin{split}
\left( \sqrt{1 - \frac{2\alpha_0}{d-1}} \sqrt{1 - \frac{2\alpha_1}{d-1}} \right)^2 &= 1 - \frac{2(\alpha_0 + \alpha_1)}{d-1} + \frac{4\alpha_0 \alpha_1}{(d-1)^2} \\
&\leq 1 - \frac{2(\alpha_0 + \alpha_1)}{d-1} + \left(\frac{\alpha_0 + \alpha_1}{d-1}\right)^2 \\
& = \left(1 - \frac{\alpha_0 + \alpha_1}{d-1} \right)^2,
\end{split}
\end{equation}
which,
together with
$\frac{\alpha_0 + \alpha_1}{d-1} \leq 1$,
proves the upper bound.
For the lower bound, observe that
\begin{equation}
\begin{split}
\left( 1 -\frac{\alpha_0 + \alpha_1}{d-1} - 2 \left(\frac{\alpha_0 - \alpha_1}{d-1}\right)^2 \right)^2 &= \left( \sqrt{1 - \frac{2 \alpha_0}{d-1}} \sqrt{1 - \frac{2 \alpha_1}{d-1}} \right)^2 \\
& + \left(\frac{\alpha_0 - \alpha_1}{d-1}\right)^2\left[ \frac{4 (\alpha_0 + \alpha_1)}{d-1} + \frac{4(\alpha_0 - \alpha_1)^2}{(d-1)^2} -3 \right].
\end{split}
\end{equation}
Now for $d \geq 4$, $0 < \alpha_0, \alpha_1 < 1/4$,
we have
$ \frac{4 (\alpha_0 + \alpha_1)}{d-1} + 4 \left(\frac{\alpha_0 - \alpha_1}{d-1}\right)^2 < 3$ and the lemma is proved.
\end{proof}
\end{lemma}

\begin{lemma}
\label{lemma:control-main-square-root}
For $d \geq 4, 0 < \alpha < 1/8, 0 < \eps < 1/2$ and
$\alpha_0 = \alpha (1 - \eps), \alpha_1 = \alpha (1 + \eps), r = \sqrt{(1 - \alpha_0)(1 - \alpha_1)}$, we have
\begin{equation}
\begin{split}
\label{eq:control-main-square-root}
\sqrt{(r-1)^2 + \frac{4 \alpha_0 \alpha_1 + (\alpha_0 + \alpha_1)(r-1)}{d-1} + \frac{\alpha_0 \alpha_1}{(d-1)^2}} \geq (1 - r) + \left[ \frac{4 \alpha_0 \alpha_1}{1 - r} - \frac{3}{2} (\alpha_0 + \alpha_1) \right]\frac{1}{d-1}
.
\end{split}
\end{equation}
\begin{proof}
Squaring \eqref{eq:control-main-square-root}, the claim will follow immediately from
\begin{equation}
\label{equation:control-main-square-root-sufficient-condition}
\begin{split}
2(\alpha_0 + \alpha_1)(1 - r) + \frac{\alpha_0 \alpha_1}{d-1} \geq \left( \frac{4 \alpha_0 \alpha_1}{1 - r} - \frac{3}{2} (\alpha_0 + \alpha_1) \right)^2 \frac{1}{d-1} + 4 \alpha_0 \alpha_1
.
\end{split}
\end{equation}
Fix $\alpha_0, \alpha_1$
and define
$$\varphi(t) = 2(\alpha_0 + \alpha_1)(1 - r) - 4 \alpha_0 \alpha_1 + t \left[ \alpha_0 \alpha_1 - \left( \frac{4 \alpha_0 \alpha_1}{1 - r} - \frac{3}{2} (\alpha_0 + \alpha_1) \right)^2 \right],
\qquad
t \in (0, +\infty)
.
$$
The sign of $\frac{d \varphi(t)}{dt}$ is the
same as that of
$\alpha_0 \alpha_1 - \left( \frac{4 \alpha_0 \alpha_1}{1 - r} - \frac{3}{2} (\alpha_0 + \alpha_1) \right)^2$.
Since $\eps < \frac{1}{2}$,
it is straightforward to verify
that $4(1 - \eps)(1 + \eps) - 3 \leq \sqrt{(1 + \eps)(1 - \eps)}$.
Additionally, $r^2 = 1 - 2 \alpha + \alpha^2(1 + \eps)(1 - \eps) \leq \left(1 - \alpha \right)^2$, so that
$\frac{1}{1-r} \leq \frac{1}{\alpha}$, and $\frac{4(1- \eps)(1 + \eps) \alpha}{1 - r} - 3 \leq \sqrt{(1 + \eps)(1 - \eps)}$.
Squaring this inequality yields,
for our
range of parameters, $\frac{d \varphi(t)}{dt} \geq 0$,
and so
$\varphi$ is minimized at
$t = 0$, and $\varphi(t) \geq \varphi(0) = 2 (\alpha_0 + \alpha_1)(1 - r) - 4 \alpha_0 \alpha_1$.

By the AM-GM inequality,
$\sqrt{\alpha_0 \alpha_1} \leq \frac{\alpha_0 + \alpha_1}{2}
\implies
(1- \alpha_0)(1 - \alpha_1) \leq \left(1 - \frac{\alpha_0 + \alpha_1}{2} \right)^2$,
and so $r \leq 1 - \frac{\alpha_0 + \alpha_1}{2}$,
and another application of AM-GM yields
$1 - r \geq 2\frac{\alpha_0 \alpha_1}{\alpha_0 + \alpha_1}$.
We conclude that
$\varphi(0) \geq 0$,
which proves the claim.
\end{proof}
\end{lemma}

\begin{lemma}
\label{lemma:control-micro}
For $d \geq 4, 0 < \alpha < 1/8, 0 < \eps < 1/2$
and
$\alpha_0 = \alpha (1 - \eps), \alpha_1 = \alpha (1 + \eps), r = \sqrt{(1 - \alpha_0)(1 - \alpha_1)}$,
we have
\begin{equation}
\begin{split}
\alpha_0 + \alpha_1 - 2\frac{\alpha_0 \alpha_1}{1-r} \leq 4 \eps^2 \alpha
.
\end{split}
\end{equation}
\begin{proof}
Observe that
$\alpha (\alpha \eps^2 + 2) \leq 1$
holds
for our assumed
range of parameters,
which implies
$2(1 - \alpha) - \eps^2 \alpha^2 \geq 1$, and further
\begin{equation}
\begin{split}
(1 - \alpha -\alpha^2 \eps^2)^2 = (1 - \alpha)^2 - \alpha^2 \eps^2 [2(1 - \alpha) - \eps^2 \alpha^2] \leq (1 - \alpha)^2 - \alpha^2 \eps^2 = r^2.
\end{split}
\end{equation}
As a consequence, $\frac{1}{1-r} \geq \frac{1}{\alpha(1 + \alpha \eps^2)} \geq \frac{1 - \alpha \eps^2}{\alpha}$, and 
\begin{equation}
\begin{split}
\alpha_0 + \alpha_1 - 2\frac{\alpha_0 \alpha_1}{1-r} \leq 2 \alpha - 2 \alpha(1 + \eps)(1 - \eps)(1 - \alpha \eps^2) = 2 \alpha \eps^2 (1 + \alpha - \alpha \eps^2 ) \leq 4 \alpha \eps^2.
\end{split}
\end{equation}
\end{proof}
\end{lemma}

\begin{lemma}
\label{lemma:control-rho}
For $d \geq 4, 0 < \alpha < 1/8, 0 < \eps < 1/2$
$\rho$ as defined in \eqref{equation:rho-definition},
and
$\alpha_0 = \alpha (1 - \eps), \alpha_1 = \alpha (1 + \eps), r = \sqrt{(1 - \alpha_0)(1 - \alpha_1)}$,
we have
\begin{equation}
\begin{split}
\rho \geq 1 - \left[ (\alpha_0 + \alpha_1) - \frac{2 \alpha_0 \alpha_1}{1 - r} \right] \frac{1}{d-1} - \frac{1}{2}\left(\frac{\alpha_0 - \alpha_1}{d-1}\right)^2 \geq 1 - 6 \frac{\alpha \eps^2}{d-1}.
\end{split}
\end{equation}
\begin{proof}
\begin{equation}
\begin{split}
2 \rho &= (r+s+1/2)+\left((r-1/2)^2 + s^2 - 2(r-1/2)s + \frac{4\alpha_0 \alpha_1}{d-1} \right)^{1/2}\\
&\geq (r+s+1/2)+\bigg((r-1/2)^2 + \left(1/2 - \frac{\alpha_0}{d-1}\right) \left(1/2 - \frac{\alpha_1}{d-1} \right) \\
&- 2(r-1/2)\left(1 - \frac{\alpha_0 + \alpha_1}{d-1} \right) + \frac{4\alpha_0 \alpha_1}{d-1} \bigg)^{1/2} \\
&= (r+s+1/2)+\left((r-1)^2 + \frac{4 \alpha_0 \alpha_1 + (\alpha_0 + \alpha_1)(r-1)}{d-1} + \frac{\alpha_0 \alpha_1}{(d-1)^2}\right)^{1/2},\\
\end{split}
\end{equation}
where the inequality is due to Lemma~\ref{lemma:control-secondary-square-root}. Invoking Lemmas~\ref{lemma:control-secondary-square-root}, and \ref{lemma:control-main-square-root}, we have
\begin{equation}
\begin{split}
2 \rho &\geq r + \frac{1}{2} \left( 1 - \frac{\alpha_0 + \alpha_1}{d-1} - 2 \left(\frac{\alpha_0 - \alpha_1}{d-1}\right)^2 \right) + \\
 & 1/2 + (1 - r) + \left[ \frac{4 \alpha_0 \alpha_1}{1 - r} - \frac{3}{2} (\alpha_0 + \alpha_1) \right]\frac{1}{d-1}, \\
\rho &\geq 1 - \left[ (\alpha_0 + \alpha_1) - \frac{2 \alpha_0 \alpha_1}{1 - r} \right] \frac{1}{d-1} - \frac{1}{2} \left( \frac{\alpha_0 - \alpha_1}{d-1} \right)^2.
\end{split}
\end{equation}
Finally,
Lemma~\ref{lemma:control-micro}
implies a lower bound on $\rho$:
\begin{equation}
\begin{split}
\rho \geq 1 - 4 \frac{\alpha \eps^2}{d-1} - 4 \frac{\alpha^2 \eps^2}{2(d-1)^2} \geq  1 - 6 \frac{\alpha \eps^2}{d-1}.
\end{split}
\end{equation}
\end{proof}
\end{lemma}

\subsection{Analysis of the empirical procedure}

Before delving into the proofs, we provide some preliminary motivation and analysis
behind
the main
estimation
procedure in Algorithm~\ref{algorithm:pseudo-spectral-gap-estimator}.

Although we only have access to a single path from $\mc$,
by considering
skipped chains with offsets $r \in \set{0, \dots, k -1}$
we effectively obtain
$k$ different paths from which to estimate $\mc^k$.
Define
the averaged estimator
\begin{equation}
\label{equation:definition-average-sg-k-estimator}
\begin{split}
\widetilde{\sg}^\dagger_{k, \alpha} (X_1, \dots, X_m) \eqdef \frac{1}{k}\sum_{r = 0}^{k-1} \widehat{\sg}_{k,r, \alpha}^\dagger \left(X_{t}^{(k,r)}, 1 \leq t \leq \floor{(m-r)/k} \right),
\end{split}
\end{equation}
where $\widehat{\sg}_{k,r, \alpha}^\dagger$ is the spectral gap of the multiplicative reversiblization of the $\alpha$-smoothed empirical transition matrix constructed from $X_{t}^{(k,r)}$:
\begin{equation}
\begin{split}
\label{equation:definition-smoothed-estimators}
\estmc^{(k, r, \alpha)}(i, j) \eqdef \frac{N_{i j}^{(k,r)} + \alpha}{N_i^{(k, r)} + d \alpha}, \qquad \estpi^{(k, r, \alpha)}_i \eqdef \frac{N_{i}^{(k,r)} + d \alpha}{\floor{(m -r) / k} + d^2 \alpha}.
\end{split}
\end{equation}
Notice that
$\estpi^{(k, r, \alpha)}_i$
requires more aggressive
smoothing
than for the transition matrix in order
to ensure stationarity.
We now derive the smoothed empirical form of \eqref{eq:LktLk} for a generic skipping rate $k$ and offset $r$:
\begin{equation}
\begin{split}
\left[\left(\widehat{\Lsym}^{(k,r, \alpha)}\right) \trn \widehat{\Lsym}^{(k,r, \alpha)} \right]_{(i,j)} = \frac{1}{\sqrt{\left(N_i^{(k,r)} + d \alpha \right) \left( N_j^{(k,r)} + d \alpha \right)}} \sum_{\ell = 1}^{d} \frac{\left(N^{(k,r)}_{\ell i} + \alpha \right) \left(  N^{(k,r)}_{\ell j} + \alpha \right)}{N_\ell^{(k,r)} + d \alpha},
\end{split}
\end{equation}
where $N_i^{(k, r)}, N_{i j}^{(k, r)}$ are defined in (\ref{equation:def-number-of-visits}, \ref{equation:def-number-of-transitions}).
The expression can be alternatively be written in its vectorized form:
\begin{equation}
\begin{split}
\left(\widehat{\Lsym}^{(k,r, \alpha)}\right) \trn \widehat{\Lsym}^{(k,r, \alpha)}  &= \left( \mathbf{D}_N^{(k,r,\alpha)} \right)^{-1/2} \left(\mathbf{N}^{(k ,r, \alpha)}\right) \trn \left( \mathbf{D}_N^{(k,r,\alpha)} \right)^{-1} \mathbf{N}^{(k ,r, \alpha)} \left( \mathbf{D}_N^{(k,r,\alpha)} \right)^{-1/2} \\
\text{with } \mathbf{N}^{(k,r,\alpha)} &\eqdef \left[N_{i j}^{(k ,r)} + \alpha \right]_{(i,j)}, \; \mathbf{D}_N^{(k,r,\alpha)} \eqdef \diag \left(N_1^{(k,r)} + d \alpha, \dots, N_d^{(k,r)} + d \alpha \right).
\end{split}
\end{equation}

\subsubsection{Proof of Theorems~\ref{theorem:pssg-confidence-intervals} and \ref{theorem:pssg-confidence-intervals-full}}
We begin with a decomposition very similar to the one employed in
Section~\ref{section:pseudo-sg-ub}.
From
the definition of $\widetilde{\sg}^\dagger_{k, \alpha}$ in (\ref{equation:definition-average-sg-k-estimator}),
it follows
that for all $K \in \mathbb{N}$,
we have
\begin{equation}
\begin{split}
\abs{\estpssgK^{(\alpha)} - \pssg} &\leq \frac{1}{K} + 2 \max_{k \in [K]} \set{\frac{1}{k^2} \sum_{r = 0}^{k - 1} \left( \nrm{\EE{\mc}^{(k, r, \alpha)}}_2 + 2 \nrm{\EE{\bpi}^{(k,r, \alpha)}}_2 + \nrm{\EE{\bpi}^{(k,r, \alpha)}}_2^2 \right) }, \\
\text{ where } &\EE{\mc}^{(k, r, \alpha)} \eqdef \left(\estDpi^{(k,r, \alpha)} \right)^{1/2} \left(\estmc^{(k,r, \alpha)} - \mc^k\right) \left(\estDpi^{(k,r, \alpha)} \right)^{-1/2}, \\
&\EE{\bpi,1}^{(k ,r, \alpha)} = \left(\estDpi^{(k, r, \alpha)}\right)^{1/2}\Dpi^{-1/2} - \identity, \qquad \EE{\bpi,2}^{(k, r, \alpha)} = \Dpi^{1/2}\left(\estDpi^{(k, r, \alpha)}\right)^{-1/2} - \identity, \\
\end{split}
\end{equation}
and $\estmc^{(k,r, \alpha)}, \estDpi^{(k, r, \alpha)}$ are the $\alpha$-smoothed estimators for $\mc^k$ and $\Dpi$ constructed from the sample path $X_t^{(k, r)}$. 

\citet[Section~3.3]{cho2001comparison} prove the perturbation
bound
$$\nrm{\estbpi - \bpi}_{\infty} \leq \hat{\kappa} \nrm{\estmc - \mc}_{\infty},$$
where $\hat{\kappa} \eqdef \frac{1}{2} \max_{j \in [d]} \set{\estginv_{j,j} - \min_{i \in [d]} \set{\estginv_{i,j}}}$ and $\estginv$ is the empirical (Drazin) \emph{group inverse} of $\estmc$ \citep{meyer1975role}.
Since computing
$\hat{\kappa}$
via matrix inversion 
is computationally expensive,
we show in Lemma~\ref{lemma:non-reversible-kappa-control} that
\begin{equation}
\begin{split}
\nrm{\estbpi^{(k,r, \alpha)} - \bpi}_\infty \leq \frac{C_\mathcal{K}}{\pssg\left(\estmc^{(k,r, \alpha)}\right)} \ln \left( 2 \sqrt{\frac{2 (\floor{(m -r) / k} + d^2 \alpha)}{\Nmin^{(k,r)} + d \alpha}} \right) \nrm{\estmc^{(k,r, \alpha)} - \mc^k}_\infty,
\end{split}
\end{equation}
where $C_\mathcal{K}$
is the universal constant from Lemma~\ref{lemma:control-visits-all-k-skipped-chains}.
From the triangle inequality,
\begin{equation}
\begin{split}
&\nrm{\estmc^{(k,r, \alpha)}(i, \cdot) - \mc^k(i, \cdot)}_1 \\
&\leq \frac{N_i^{(k,r)}}{N_i^{(k,r)} + d \alpha} \nrm{\estmc^{(k,r)}(i, \cdot) - \mc^k(i, \cdot)}_1 + \frac{\alpha d \nrm{ (1/d) \cdot \unit - \mc^k(i, \cdot)}_1}{N_i^{(k,r)} + d \alpha},
\end{split}
\end{equation}
and writing
\begin{equation}
\begin{split}
\hat{d}^{(k,r, \alpha)} &\eqdef 4 \tau_{\delta/d, \floor{(m - r)/ k }} \sqrt{\frac{d }{\Nmin^{(k,r)} + d \alpha}} + \frac{ 2 \alpha d}{\Nmin^{(k,r)} + d \alpha}, \\
\hat{b}^{(k,r, \alpha)} &\eqdef \frac{C_\mathcal{K}}{\pssg\left(\estmc^{(k,r, \alpha)}\right)} \ln \left( 2 \sqrt{\frac{2 (\floor{(m -r) / k} + d^2 \alpha)}{\Nmin^{(k,r)} + d \alpha }} \right) \hat{d}^{(k,r, \alpha)} ,
\end{split}
\end{equation}
it follows via
Lemma~\ref{lemma:mc-empirical-learning-infinity-norm}
that $\nrm{\estbpi^{(k,r,\alpha)} - \bpi}_{\infty} \leq \hat{b}^{(k,r,\alpha)}$.

Define $\hat{c}^{(k,r,\alpha)} \eqdef \frac{1}{2} \max \bigcup_{i \in [d]} \set{\frac{\hat{b}^{(k,r,\alpha)}}{\estbpi^{(k,r,\alpha)}(i)}, \frac{\hat{b}^{(k,r,\alpha)}}{\left[\estbpi^{(k,r,\alpha)}(i) - \hat{b}^{(k,r,\alpha)}\right]_+}}$, and recall from \citet{hsu-mixing-ext},
that
$\nrm{\estbpi^{(k,r,\alpha)} - \bpi}_\infty \leq \hat{b}^{(k,r,\alpha)}
\implies
\nrm{\EE{\bpi}^{(k,r,\alpha)}}_2 \leq \hat{c}^{(k,r,\alpha)}$.
By sub-multiplicativity of the spectral norm and norm properties of diagonal matrices, we have
\begin{equation}
\begin{split}
\nrm{\EE{\mc}^{(k, r, \alpha)}}_2 \leq \frac{\Nmax^{(k,r)} + d \alpha}{ \floor{(m-r)/k} + d^2 \alpha} \sqrt{d} \nrm{ \estmc^{(k, r, \alpha)} - \mc^k }_\infty \frac{ \floor{(m-r)/k} + d^2 \alpha}{\Nmin^{(k,r)} + d \alpha}.
\end{split}
\end{equation}
Putting
\begin{equation}
\begin{split}
\hat{a}^{(k,r,\alpha)} \eqdef \sqrt{d} \frac{\Nmax^{(k,r)} + d \alpha}{\Nmin^{(k,r)} + d \alpha} \hat{d}^{(k,r, \alpha)},
\end{split}
\end{equation}
we have that
$\nrm{\EE{\mc}^{(k, r, \alpha)}}_2 \leq \hat{a}^{(k,r,\alpha)}$
holds with probability at least $1- \delta$.

Turning to the simpler reversible case,
consider the $\alpha$-smoothed version of the estimator $\frac{1}{2} \left( \estmc \rev + \estmc \right)$, where the $(i,j)$th
entry is $\frac{N_{i j} + N_{j i}}{2 N_i}$. Reversibility
allows us to apply Weyl's inequality:
\begin{equation}
\begin{split}
  \abs{\lambda_i\left( \frac{1}{2} \left( \estmc \rev + \estmc \right) \right) - \lambda_i \left( \mc \right)} &= \abs{\lambda_i \left( \frac{1}{2} \left( \widehat{\Lsym} \trn + \widehat{\Lsym} \right) \right) - \lambda_i \left( \Lsym \right)} \leq \nrm{ \widehat{\Lsym} - \Lsym }_2,
  \qquad i\in[d].
\end{split}
\end{equation}
The interval widths can now
be deduced from Corollary~\ref{corollary:learning-time-reversal-empirical-bounds} and Corollary~\ref{corollary:reversible-kappa-control}.

\subsubsection{Asymptotic interval widths}

\paragraph{Non-reversible setting.}

The definition of the pseudo-spectral gap
implies that
$\pssg(\mc^k) \geq k \pssg(\mc)$, and
assuming a smoothing parameter $\alpha < 1/d$,
a straightforward computation (ignoring logarithmic factors) yields
\begin{equation}
\begin{split}
\sqrt{m} \hat{a}^{(k,r,\alpha)} &= \tilde{\bigO} \left(  \frac{\sqrt{k} \nrm{\mc}_{\bpi} d}{\sqrt{\pimin}} \right), \\
 \sqrt{m}  \hat{b}^{(k,r,\alpha)} &\leq \tilde{\bigO} \left( \frac{\sqrt{d}}{\sqrt{k} \pssg \sqrt{\pimin} } \right), \\
 \sqrt{m}  \hat{c}^{(k,r,\alpha)} &\leq \tilde{\bigO} \left( \frac{\sqrt{d}}{\sqrt{k}\pssg \pimin^{3/2}} \right).
\end{split}
\end{equation}
It follows that
\begin{equation}
\begin{split}
\abs{\estpimin^{(\alpha)} - \pimin} &= \tilde \bigO \left(  \frac{\sqrt{d}}{\pssg \sqrt{\pimin m} } \right), \\
 \abs{\estpssgK^{(\alpha)} - \pssg} &= \tilde{\bigO} \left( \frac{1}{K} + \sqrt{\frac{d}{\pimin m}}  \left( \sqrt{d} \nrm{\mc}_{\bpi}  + \frac{1}{\pssg \pimin} \right) \right).
\end{split}
\end{equation}

\paragraph{Reversible setting.}
Here,
\begin{equation}
\begin{split}
\sqrt{m} \hat{a}^{(\alpha)} = \tilde{\bigO} \left(  \frac{\nrm{\mc}_{\bpi} d}{\sqrt{\pimin}} \right) ,\qquad \sqrt{m}  \hat{b}^{(\alpha)} = \tilde{\bigO} \left( \frac{\sqrt{d}}{\asg \sqrt{\pimin} } \right), \qquad \sqrt{m}  \hat{c}^{(\alpha)} = \tilde{\bigO} \left( \frac{\sqrt{d}}{\asg \pimin^{3/2}} \right),
\end{split}
\end{equation}
so that
\begin{equation}
\begin{split}
  \abs{\estpimin^{(\alpha)} - \pimin} = \tilde{\bigO} \left( \frac{\sqrt{d}}{\asg \sqrt{\pimin m} } \right), \qquad   \abs{\estasg^{(\alpha)} - \asg} = \tilde{\bigO} \left( \sqrt{\frac{d}{\pimin m}}  \left( \sqrt{d} \nrm{\mc}_{\bpi} + \frac{1}{\asg \pimin} \right) \right).
\end{split}
\end{equation}

\begin{remark}
  \label{remark:improvement-empirical-intervals}
  Let us compare the intervals obtained in \citet[Theorem 4.1]{hsu-mixing-ext} for the reversible case with ours.
  For estimating $\pimin$,
  we obtain an improvement of a factor of $\tilde \bigO \left( \sqrt{d} \right)$,
  and for $\pssg$, we witness a similar improvement when
  $\tmix = \Omega \left( \sqrt{d} \max_{i \in [d]} \bpi(i) \right)$
  --- i.e. in the case where the chain is not rapidly mixing.
  When comparing to \citet[Theorem 4.2]{hsu-mixing-ext}, which intersects the point-estimate intervals
  with the empirical ones, our results are asymptotically equivalent. However, this is a somewhat misleading comparison,
  since the aforementioned intersection, being asymptotic in nature, is oblivious to the rate decay of the empirical intervals.
\end{remark}

\subsubsection{Computational complexity in the reversible setting}

The best current time complexity of multiplying (and inverting, and diagonalizing) $d\times d$ matrices is
$\bigO(d^\omega)$ where $2 \leq \omega \leq 2.3728639$ \citep{le2014powers}.

\paragraph{Time complexity.}

In the reversible case, the time complexity of our algorithm is $\bigO \left( m + d^2 + \mathcal{C}_{\lambda_\star} \right)$, where $\mathcal{C}_{\lambda_\star}$ is the complexity of computing the second largest-magnitude eigenvalue of a symmetric matrix.
For this task, we consider
here the Lanczos algorithm.
Let $\lambda_1, \ldots, \lambda_d$ be the eigenvalues of a symmetric real matrix ordered by magnitude,
and
denote by
$\lambda_1^{\textrm{{\tiny \textup{LANCZOS}}}}$
the algorithm's approximation for $\lambda_1$.
Then,
for a
stochastic
matrix, it is known \citep{kaniel1966estimates, paige1971computation, saad1980rates}
that
\begin{equation}
\begin{split}
\abs{\lambda_1 - \lambda_1^{\textrm{{\tiny \textup{LANCZOS}}}}} \leq C_\textrm{{\tiny \textup{LANCZOS}}} R^{-2(n-1)},
\end{split}
\end{equation}
where $C_\textrm{{\tiny \textup{LANCZOS}}}$ a universal constant,
$n$ is the number of iterations
(in practice often $n \ll d$),
$R = 1 + 2 r + 2 \sqrt{r^2 + r}$,
and $r = \frac{\lambda_1 - \lambda_2}{\lambda_2 - \lambda_d}$.
So in order to attain additive accuracy
$\eta$,
it suffices iterate the method
$n \geq 1 + \frac{1}{2} \frac{ \ln \left( C \eta^{-1} \right)}{\ln (R)} = \bigO \left( \frac{\ln \left( \eta^{-1} \right)}{\ln (R)} \right)$ times.

A single
iteration
involves multiplying a vector by a matrix,
incurring a cost of
$\bigO \left( d^2 \right)$,
and so the full complexity of the
Lanczos algorithm
is
$\bigO\left( m + d^2 \frac{\ln \left( \eta^{-1} \right)}{\ln (R)} \right)$.
More refined complexity analyses may be found in
\citet{kuczynski1992estimating, arora2005fast}

The previous
approach of
\citet{hsu-mixing-ext}
involved computing the Drazin inverse via matrix inversion,
incurring a cost of $\bigO \left( m + d^\omega \right)$.
Thus, our proposed computational method is faster over a non-trivial regime.

\paragraph{Space complexity.}
When estimating the absolute spectral gap of a reversible case, the space complexity remains $\bigO \left( d^2 \right)$, as the full empirical transition matrix is being constructed. 

\subsubsection{Computational complexity in the non-reversible setting}

\paragraph{Time complexity.}

For the non-reversible setting, we are estimating multiplicative reversiblizations of powers of chains
which involves matrix multiplication.
Our time complexity is
$\bigO \left( K^2 \left(m + d^\omega\right) \right)$.

\paragraph{Space complexity.}
In the non-reversible case,
$\bigO \left( K^2 \right)$ matrices are being constructed,
so that the overall space complexity is $\bigO \left( K^2 d^2 \right)$.

\paragraph{Reducing computation of a pseudo-spectral gap to the computation of spectral radii.}
\label{section:reduction-to-radii}
We now show that,
in order to compute
$\pssgK$
it suffices to
compute the value of $K$ spectral radii.
For a Markov chain $\mc$ with stationary distribution $\bpi$ and corresponding $\Lsym$,
we observe that
$\sg((\mc^k) \rev \mc^k) = \sg((\Lsym^k) \trn \Lsym^k)$
and compute,
\begin{equation}
\label{eq:LktLk}
\begin{split}
  \left[\left(\Lsym^k\right) \trn \Lsym^k\right]_{(i,j)} = \sum_{\ell = 1}^{d} \sqrt{\frac{\pi_\ell}{\pi_i}} \mc(\ell, i)^k \sqrt{\frac{\pi_\ell}{\pi_j}} \mc(\ell,j)^k,
\qquad k \in \mathbb{N}
  .
\end{split}
\end{equation}
Since $(\Lsym^k) \trn \Lsym^k$ is a symmetric positive semi-definite matrix,
its eigenvalues may be ordered
such that $\nu_1 \geq \nu_2 \geq \dots \nu_d \geq 0$.
From \eqref{eq:LktLk}
it follows
that $\sqrt{\bpi}$ is a left eigenvector for eigenvalue $1$,
and by the Perron-Frobenius theorem, $\rho((\Lsym^k) \trn \Lsym^k) = 1$.
By symmetry, we can express $(\Lsym^k) \trn \Lsym^k$ over an {orthogonal} left row eigen-basis
$v_1, \dots, v_d$ with the associated eigenvalues
$\nu_1, \dots, \nu_d$, where $(\nu_1, v_1) = (1, \sqrt{\bpi})$:
\begin{equation}
\begin{split}
\left(\Lsym^k\right) \trn \Lsym^k = \sum_{i = 1}^{d} \nu_i (v_i \trn v_i) \text{ hence } \rho\left( \left(\Lsym^k \right) \trn \Lsym^k - \sqrt{\bpi} \trn \sqrt{\bpi} \right) = \nu_2.
\end{split}
\end{equation}
We have thus shown that at the cost of computing $\sqrt{\bpi} \trn \sqrt{\bpi}$
does not depend on $k \in [K]$,
reducing the problem
to computing the largest eigenvalue of a real symmetric matrix.
The latter may be 
achieved via the {Lanczos method},
and is more efficient than
the {power iteration methods}
in the symmetric case.
In summary, we have shown that
\begin{equation}
\begin{split}
\pssgK(\mc) = \max_{k \in [K]} \set{ \frac{1 - \rho \left( \left(\Lsym^k \right) \trn \Lsym^k - \sqrt{\bpi} \trn \sqrt{\bpi} \right)}{k}} 
\end{split}
.
\end{equation}

\paragraph{Discussion.}

Recall also that for any $k \in \mathbb{N}$, $\frac{\sg^\dagger_k}{k}$
is a lower bound for the pseudo-spectral gap, hence an upper one for the mixing time,
so that any $k$ would yield a usable and conservative value for it. In practice this implies that the procedure makes iterative improvements to its known value of $\tmix$.

\section{Additional auxiliary results}

\subsection{Pseudo-spectral gap of reversible chains.}

An immediate consequence of \eqref{eq:control-tmix-reversible} and \eqref{eq:control-tmix-non-reversible} is that in the reversible case,
\begin{equation}
\frac{\asg}{2 \ln \left(4/ \pimin \right)} \leq \pssg \leq \frac{\asg}{1 - \asg} \left( \frac{\ln\left( 1 / \pimin \right) + 1}{\ln{2}} +2\right),
\end{equation}
i.e. $\pssg = \tilde \Theta (\asg)$.
Lemma~\ref{lemma:pssg-reversible}
below establishes the stronger fact that
for
reversible Markov chains, the {absolute spectral gap} and the {pseudo-spectral gap} are within a multiplicative
factor of $2$;  moreover, the {pseudo-spectral gap} has a closed form in terms of the {absolute spectral gap}:
\begin{lemma}
  \label{lemma:pssg-reversible}
  For reversible
  $\mc$,
we have
\begin{equation}
\asg(\mc) \leq \pssg(\mc) = \asg(\mc)[2 -  \asg(\mc)] \leq 2\asg(\mc).
\end{equation}

\begin{proof}
  Reversibility implies
  $\mc \rev = \mc$, and so $\pssg = \max_{k \geq 1} \set{ \frac{\sg(\mc^{2k})}{k}}$.
  Denoting by $1 = \lambda_1 > \lambda_2 \geq \dots \geq \lambda_d$ the eigenvalues of $\mc$, we have that for all $i \in [d]$ and $k \geq 1$, $\lambda_i^{2k}$ is an eigenvalue for $\mc^{2k}$, and furthermore $\lambda_\star^{2k}$ with $\lambda_\star = \max \set{\lambda_2(\mc), \abs{\lambda_d(\mc)}}$ is necessarily the second largest. We claim that
\begin{equation}
\pssg = \max_{k \geq 1} \left\{ \frac{1 - \lambda_\star^{2k}}{k} \right\}= 1 - \lambda_\star^{2}
\end{equation}
--- that is, the maximum is achieved at $k=1$. Indeed, $1 - \lambda_\star^{2k} =  (1 - \lambda_\star^2)\left( \sum_{i = 0}^{k-1} \lambda_\star^i \right)$ and the latter sum is at most $k$ since $\lambda_\star <1$. As a result, $\pssg(\mc) = 1 - \lambda_\star^2 = 1 - (1 - \asg(\mc))^2 = \asg(\mc)[2 -  \asg(\mc)]$ and also $\asg(\mc) \leq \pssg(\mc) \leq 2\asg(\mc)$.
\end{proof}
\end{lemma}

\subsection{Controlling the number of visits to all \texorpdfstring{$k$}{k}-skipped \texorpdfstring{$r$}{r}-offset chains.}

The following Lemma~\ref{lemma:control-visits-all-k-skipped-chains} is instrumental in proving Theorem~\ref{theorem:pseudo-sg-ub}.
It quantifies the trajectory length
sufficient to guarantee that
that all states of an ergodic Markov chain with skipping rate $k\in[K]$
have been visited approximately according to their expected value with high confidence.
Recall that $\kps$, defined immediately following \eqref{definition:pssg},
is the smallest positive integer
such that $\pssg = {\sg \left( \left(\mc \rev \right)^\kps \mc ^\kps\right)}/{\kps}$.

\begin{lemma}
\label{lemma:control-visits-all-k-skipped-chains}
For $(d, K) \in \mathbb{N}^2$,  let $X_1, \dots, X_m \sim (\mc, \bmu)$ be a $d$-state time homogeneous
ergodic Markov chain with pseudo-spectral gap $\pssg$ and stationary distribution $\bpi$ minorized by $\pimin$.
For $ 0 < \eta \leq 1$, consider the event  
$$\mathcal{K}^K_\eta = \set{
\max_{(k,i)\in[K]\times[d]}
  \abs{N_i^{(k)} - \E[\mc, \bpi]{N_i^{(k)}}} \leq \eta \E[\mc, \bpi]{N_i^{(k)}}},$$
where $N_i^{(k)}$ is defined in (\ref{equation:def-number-of-visits}).
Then, for $m \geq \frac{C_\mathcal{K} K^2}{\pimin \eta^2 \pssg} \ln\left( \frac{\sqrt{2 \pimin^{-1}} d K }{\delta} \right)$,
we have that
$\mathcal{K}^K_\eta$ holds with probability at least $1 - \delta$, where $C_\mathcal{K} \le 192$ is a universal constant.
\end{lemma}
\begin{proof}
By two applications of the union bound, followed by \citet[Proposition~3.10]{paulin2015concentration}, we have
\begin{equation}
\begin{split}
\PR[\mc, \bmu]{\compt{\mathcal{K}^K_\eta}} & \leq \sqrt{\nrm{\bmu/\bpi}_{2, \bpi}} \sum_{k = 1}^{K} \sum_{i = 1}^{d}  \mathbf{P}_{\mc, \bpi}^{1/2}\left[ \abs{N_i^{(k)} - \E[\mc, \bpi]{N_i^{(k)}}} > \E[\mc, \bpi]{N_i^{(k)}} \right].
\end{split}
\end{equation}
This accounts for non-stationary starting distributions,
and it remains to upper bound each of the summands.
Fixing
$k$ and $i$,
define
$\phi_t(X_t) = \pred{
  \mydiv{k}{t}
  \text{ and } X_t = i}$,
$t \in [m]$.
Putting
$N_i^{(k)} = \sum_{t = 1}^{m} \phi_t(X_t)$,
we obtain 
$\E[\mc, \bpi]{\phi_t} = \pred{
  \mydiv{k}{t}
} \pi_i$
from the ergodic theorem for stationary chains.
Thus, for $j\in[d]$,
we have
$\abs{\phi_t(j) - \E[\mc, \bpi]{\phi_t}} \leq 1$.
We cannot directly apply \citet[Theorem~3.4]{paulin2015concentration} for reasons that we detail in
Remark~\ref{remark:known-concentration-bound-not-applicable},
and instead
derive
in Theorem~\ref{theorem:alternative-concentration-pssg}
(using his methods)
a concentration bound tailored
to our needs.

Let $\tilde{\sigma}_r$ defined as in Theorem~\ref{theorem:alternative-concentration-pssg}. For any $t$,
we compute the closed form for the variance of $\phi_t$:
\begin{equation}
\begin{split}
  \Var[\mc, \bpi]{\phi_t} = \E[\mc, \bpi]{\phi_t} \left( 1 - \E[\mc, \bpi]{\phi_t}\right) = \pred{
    \mydiv{k}{t}
  } \pi_i ( 1 - \pi_i).
\end{split}
\end{equation}
Thus, for
$1 \leq r \leq \kps$,
we have
\begin{equation}
\begin{split}
  \tilde{\sigma}_r &= \pi_i (1 - \pi_i) \abs{ \set{ 0 \leq s \leq \floor{(m - r) / \kps} :
      \mydiv{k}{r + s \kps}
  }} \leq \frac{\pi_i}{4 \kps} \left( m + \kps \right).
\end{split}
\end{equation}
Invoking Theorem~\ref{theorem:alternative-concentration-pssg},
\begin{equation}
\label{equation:control-visits-to-single-state}
\begin{split}
 \PR[\mc, \bpi]{\abs{N_i^{(k)} - \E[\mc, \bpi]{N_i^{(k)}}} > \eta \E[\mc, \bpi]{N_i^{(k)}}} &\leq 2 \expo{-\frac{ \floor{m/k}^2 \pi_i \eta^2  \pssg}{2 \left( m + \kps \right) + 20 \floor{m/k} \eta }} \\
 &\leq 2\expo{-\frac{ m \pi_i \eta^2  \pssg }{96 k^2}},
\end{split}
\end{equation}
where the second inequality relies on the fact that $m \geq \max \set{2K, \frac{1}{\pssg}}$,
and $\pssg = \frac{\sg^{\dagger}_\kps}{\kps} \leq \kps^{-1}$.
It follows that
\begin{equation}
\begin{split}
\PR[\mc, \bmu]{\compt{\mathcal{K}^K_\eta}} & \leq \sqrt{2 \pimin^{-1}} \sum_{k = 1}^{K} \sum_{i = 1}^{d} \expo{-\frac{ m \pi_i \eta^2  \pssg }{96 k^2}} \leq \sqrt{2 \pimin^{-1}} K d \expo{-\frac{ m \pimin \eta^2  \pssg }{192 K^2}}. \\
\end{split}
\end{equation}
\end{proof}

\begin{theorem}[Bernstein-type bound for non-reversible Markov chain.]
\label{theorem:alternative-concentration-pssg}
For $d \in \mathbb{N}$, let $X_1, \dots, X_m \sim (\mc, \bpi)$ be a $d$-state stationary time homogeneous ergodic
Markov chain with pseudo-spectral gap $\pssg$.
Suppose that
$\phi_1 ,\dots \phi_m:[d]\to\R$
are
such that
for some $C>0$
and all
$j\in[d],t\in[m]$,
we have
$\abs{\phi_t(j) - \E[\mc, \bpi]{\phi_t}} \leq C$.
Define, for $1 \leq r \leq \kps$, the quantities
$$\tilde{\sigma}_r \eqdef \sum_{s = 0}^{\floor{(m - r) / \kps}} \Var[\mc, \bpi]{\phi_{r + s  \kps}}$$
and
$$\tilde{\sigma} \eqdef \kps \max_{1 \leq r \leq \kps} \tilde{\sigma}_r,$$ where $\kps$
is
the smallest positive integer such that $\pssg = \frac{\sg^{\dagger}_\kps}{\kps}$. Then
\begin{equation}
\begin{split}
\PR[\mc, \bpi]{\abs{\sum_{t = t}^{m} \phi_t(X_t) - \E[\mc, \bpi]{\sum_{t = t}^{m} \phi_t(X_t)}} > \eps} \leq 2 \expo{-\frac{\eps^2 \pssg}{8 \tilde{\sigma} + 20 \eps C}}.
\end{split}
\end{equation}
\end{theorem}
\begin{proof}
  Our proof
proceeds along
the
lines
of \citet[Theorem~3.4]{paulin2015concentration}.
We begin by partitioning the sequence
$\left(\phi_1(X_1), \dots, \phi_m(X_m)\right)$ into $\kps$
skipped sub-sequences indexed by $1 \leq r \leq \kps$,
\begin{equation}
\begin{split}
\left(\phi(\X_{r + s \kps})\right)_{0 \leq s \leq \floor{(m - r)/\kps}},
\end{split}
\end{equation}
with
\begin{equation}
\begin{split}
\paren{\X_{r + s \kps}}_{0 \leq s \leq \floor{(m - r)/\kps}} \sim (\mc^\kps, \bpi \mc^{r-1})
\end{split}
\end{equation}
from Jensen's inequality, for any
distribution
$\boldsymbol{\alpha} = (\alpha_1, \dots, \alpha_\kps) \in \Delta_{\kps}$ and $\theta > 0$,
\begin{equation}
\begin{split}
\E[\mc, \bpi]{\expo{\theta \sum_{t = 1}^{m} \phi_t(X_t)}} \leq \sum_{r = 1}^{\kps} \alpha_i \E[\mc, \bpi]{\expo{\frac{\theta}{\alpha_i} \sum_{s = 0}^{\floor{(m - r)/\kps}} \phi_{r + s \kps}(X_{r + s \kps})}}.
\end{split}
\end{equation}
In particular for $i \in [\kps]$
and $\alpha_i = \frac{1}{\kps}$,
it follows from
\citet[(5.6), (5.7)]{paulin2015concentration} that
\begin{equation}
\begin{split}
\E[\mc, \bpi]{\expo{\theta \sum_{t = t}^{m} \phi_t(X_t)}} &\leq  \frac{1}{\kps} \sum_{r = 1}^{\kps} \expo{\frac{2 \theta^2 \kps^2 \tilde{\sigma}_r}{\sg^{\dagger}_\kps} \left( 1 - \frac{10 \theta \kps}{\sg^{\dagger}_\kps} \right)^{-1}} \\
&\leq \expo{\frac{2 \theta^2 \kps \max_{1 \leq r \leq \kps}\tilde{\sigma}_r}{\pssg} \left( 1 - \frac{10 \theta }{\pssg} \right)^{-1}}.
\end{split}
\end{equation}
From
the estimate
on its moment generating function, the random variable
$\sum_{t = t}^{m} \phi_t(X_t)$ is {sub-gamma} with {variance factor}
$\nu \eqset \frac{4 \kps \max_{1 \leq r \leq \kps}\tilde{\sigma}_r}{\pssg}$ and {scale parameter} $B \eqset \frac{10}{\pssg}$, i.e. 
$$\E[\mc, \bpi]{\expo{\theta \sum_{t = t}^{m} \phi_t(X_t)}} \leq \expo{\frac{\nu \theta^2}{2(1 - B \theta)}}.$$
An application of Markov's inequality concludes the proof.
\end{proof}

\begin{remark}
  \label{remark:known-concentration-bound-not-applicable}
  Curiously,
  \citet[Theorem~3.4]{paulin2015concentration},
  in its most general form,
does not seem to yield useful bounds for our particular problem.
Indeed, one is faced with bounding the ratio
$\frac{R(m,k,\kps)}{\kps}$ where
$R(m,k,\kps) \eqdef \cfrac{\sum_{1 \leq r \leq \kps} \sqrt{\tilde{\sigma}_r}}{\min_{1 \leq r \leq \kps} \sqrt{\tilde{\sigma}_r}}$. However, for our considered $\phi_1, \dots, \phi_m$,
the latter quantity is computed to be
$$R(m,k,\kps) = \frac{\sum_{1 \leq r \leq \kps} \sqrt{\abs{ \set{ 0 \leq s \leq \floor{(m - r) / \kps} :
        \mydiv{k}{r + s \kps}
}}}}{\min_{1 \leq r \leq \kps} \sqrt{\abs{ \set{ 0 \leq s \leq \floor{(m - r) / \kps} :
        \mydiv{k}{r + s \kps}
}}}},$$
and in particular, the the denominator
might be zero.
(For a concrete example, take $k = \kps = 5$, and $r = 1$.)
We hope that our observation is a step towards
answering \citet[Remark~3.7]{paulin2015concentration}.
\end{remark}

\subsection{Empirical bounds for learning the stationary distribution of a chain.}
\label{sec:empirical-bounds}

The technical results of this section are at the core of three different improvements. First, to extend the perturbation results of \citet{hsu2015mixing} to the non-reversible setting, second, to improve the width of their confidence intervals for $\pimin$ roughly by a factor $\tilde \bigO(\sqrt{d})$ in the reversible case through the use of a vector-valued martingale technique, and third, to reduce the computation cost of their intervals from $\bigO(d^3)$ to $\bigO(d^2)$, by showing that we can trade the computation of a pseudo-inverse with the already computed estimator for the absolute spectral gap.

\begin{lemma}[Perturbation bound for stationary distribution of ergodic Markov chain]
\label{lemma:non-reversible-kappa-control}
Let $\mc_1$ (resp. $\mc_2$)
be 
an ergodic Markov chain with stationary distribution
$\bpi_1$ (resp. $\bpi_2$) minorized by $\pimin(\mc_1)$ (resp. $\pimin(\mc_2)$)
and pseudo-spectral gap $\pssg(\mc_1)$ (resp. $\pssg(\mc_2)$).
Then
\begin{equation}
\begin{split}
\nrm{\bpi_1 - \bpi_2}_\infty \leq \frac{C_\mathcal{K}}{\max \set{ \pssg(\mc_1), \pssg(\mc_2) }} \ln \left( 2 \sqrt{\frac{2}{\max \set{ \pimin(\mc_1), \pimin(\mc_2) }}} \right) \nrm{\mc_1 - \mc_2}_\infty,
\end{split}
\end{equation}
where $C_\mathcal{K}$ is the universal constant in Lemma~\ref{lemma:control-visits-all-k-skipped-chains}.
\end{lemma}
\begin{proof}
The proof follows a similar argument as to the one at \citep[Lemma~8.9]{hsu-mixing-ext} except that \eqref{equation:control-visits-to-single-state}, valid for non-reversible Markov chains, is used instead of \citet[Chapter~12]{levin2009markov}.
\citet[Section 3.3]{cho2001comparison}
have shown established that
$\nrm{\bpi_1 - \bpi_2}_\infty \leq \kappa_1 \nrm{\mc_1 - \mc_2}_\infty$,
where
$\kappa_1 = \frac{1}{2} \max_{j \in [d]} \set{ \pi_j  \max_{i \neq j} \E[\mc_1, \delta_i]{\tau_j}}$, $\boldsymbol{\delta}_i$
is
the distribution
whose support consists of
state $i$,
and $\tau_j$
is
the first time state $j$ is visited by the Markov chain.
We proceed by with upper bounding $\E[\mc_1, \boldsymbol{\delta}_i]{\tau_j}$.
Setting
$\eta \eqset k \eqset 1$
in
\eqref{equation:control-visits-to-single-state},
we have, for $(m,i,j) \in \mathbb{N} \times [d]^2$, that
\begin{equation}
\begin{split}
\PR[\mc_1, \boldsymbol{\delta}_i]{\tau_j > m} &= \PR[\mc_1, \boldsymbol{\delta}_i]{N_j = 0} \leq \PR[\mc_1, \boldsymbol{\delta}_i]{\abs{N_j - \pi_j m} > \pi_j m}.
\end{split}
\end{equation}
Further,
since
$\nrm{\boldsymbol{\delta}_i/\bpi_1}_{2, \bpi_1} = \frac{1}{\bpi_1(i)}$
for $m \geq \frac{C_\mathcal{K}}{\bpi_1(j) \pssg(\mc_1)} \ln \left( 2 \sqrt{\frac{2}{\bpi_1(i)}} \right)$, it
follows
that $\PR[\mc_1, \boldsymbol{\delta}_i]{\tau_j > m} \leq \frac{1}{2}$.
Hence,
\begin{equation}
\begin{split}
\E[\mc_1, \boldsymbol{\delta}_i]{\tau_j} &= \sum_{\ell \geq 0} \PR[\mc_1, \boldsymbol{\delta}_i]{\tau_j > \ell} = \sum_{0 \leq \ell < m} \PR[\mc_1, \boldsymbol{\delta}_i]{\tau_j > \ell} + \sum_{k \geq 1}\sum_{k m \leq \ell < (k+1) m} \PR[\mc_1, \boldsymbol{\delta}_i]{\tau_j > \ell} \\
&\leq m + \sum_{k \geq 1}\sum_{k m \leq \ell < (k+1) m} \PR[\mc_1, \boldsymbol{\delta}_i]{\tau_j > k m}.
\end{split}
\end{equation}
By the Markov property,
$\PR[\mc_1, \boldsymbol{\delta}_i]{\tau_j > m} \leq 1/2
\implies
\PR[\mc_1, \boldsymbol{\delta}_i]{\tau_j > k m} \leq 2^{-k}$,
$k\ge1$,
and so
\begin{equation}
\begin{split}
\E[\mc_1, \boldsymbol{\delta}_i]{\tau_j} &\leq m + m \sum_{k \geq 1} 2^{-k} = m + m \frac{2^{-1} - \lim_{k \rightarrow \infty}2^{-k}}{1 - 1/2} = 2m,
\end{split}
\end{equation}
which completes the proof.
\end{proof}

The following corollary is a natural extension of a similar result of \citet{hsu2015mixing} to matrix-valued martingales.

\begin{corollary}[to Theorem~\ref{theorem:matrix-freedman}]
\label{corollary:freedman-matrix-compacted}
Consider a matrix martingale difference sequence $\set{\Y_t : t = 1, 2, 3, \dots}$ with
dimensions $d_1 \times d_2$, such that $\nrm{\Y_t}_2 \leq 1$ almost surely for $t = 1,2, \dots$. Then with $\nrm{\mathbf{\Sigma}_m}_2$ as defined in Theorem~\ref{theorem:matrix-freedman},
for all $\eps \geq 0$ and $\sigma^2 > 0$,
\begin{equation}
\begin{split}
\PR{\nrm{\sum_{t=1}^{m}\Y_t}_2 > \sqrt{2 \nrm{\mathbf{\Sigma}_m}_2 \tau_{\delta, m}} + \frac{5}{3} \tau_{\delta, m} } &\leq \delta, \\ 
\text{ where } \tau_{\delta, m} \eqdef \inf \set{t > 0: \left(1 + \ceil{\ln(2m/t)}_+\right) (d_1 + d_2)  e^{-t} \leq \delta} &= \bigO \left( \ln\left( \frac{(d_1 + d_2) \ln m}{\delta} \right) \right).\\
\end{split}
\end{equation}
\end{corollary}
\begin{proof}
  Let $m \in \mathbb{N}$. From properties of the spectral norm (sub-additivity, sub-multiplicativity, invariance under transpose),
  an application of Jensen's inequality,
  and $\nrm{\Y_t}_2 \leq 1$, we have
\begin{equation}
\begin{split}
\nrm{\W_{col, k}}_2 = \nrm{\sum_{t=1}^{k} \E[t-1]{ \Y_t \Y_t \trn}}_2 \leq \sum_{t=1}^{k} \mathbf{E}_{t-1}{ \nrm{  \Y_t \Y_t \trn}_2} = \sum_{t=1}^{k}  \mathbf{E}_{t-1}{\nrm{  \Y_t }^2_2} \leq k,
\end{split}
\end{equation}
and similarly $\nrm{\W_{\row, k}}_2 \leq k$,
concluding that
$\nrm{\mathbf{\Sigma}_k}_2 \leq k$.
Let $\eps > 0$ and $\sigma >0$.
Define
$\sigma_i^2 \eqdef \frac{e^i \eps}{2}$
for $i \in \set{0, 1, \dots, \ceil{\ln\left(2m/\eps\right)}_+}$,
and $\sigma_{-1}^2 \eqdef -\infty$.
Observing that
$\sqrt{2 \max \set{ \sigma_0^2, e \nrm{\mathbf{\Sigma}_k}_2 } \eps} \leq \sqrt{2 \nrm{\mathbf{\Sigma}_k}_2 \eps} + \eps$,
we have
\begin{equation}
\begin{split}
&\PR{\exists k \in [m], \nrm{\sum_{t=1}^{k}\Y_t}_2 > \sqrt{2 \nrm{\mathbf{\Sigma}_k}_2 \eps} + \frac{5}{3}\eps } \\
&\leq \sum_{i=0}^{\ceil{\ln(2m/\eps)}_+} \PR{\exists k \in [m], \nrm{\sum_{t=1}^{k}\Y_t}_2 > \sqrt{2 \max \set{ \sigma_0^2, e \nrm{\mathbf{\Sigma}_k}_2  }\eps} + \frac{2}{3}\eps \text{ and } \nrm{\mathbf{\Sigma}_k}_2 \in (\sigma_{i-1}^2, \sigma_{i}^2] } \\
&\leq \sum_{i=0}^{\ceil{\ln(2m/\eps)}_+} \PR{\exists k \in [m], \nrm{\sum_{t=1}^{k}\Y_t}_2 > \sqrt{2 \max \set{ \sigma_0^2, e \sigma_{i-1}^2 } \eps} + \frac{2}{3}\eps \text{ and } \nrm{\mathbf{\Sigma}_k}_2 \in (\sigma_{i-1}^2, \sigma_{i}^2] } \\
&\leq \sum_{i=0}^{\ceil{\ln(2m/\eps)}_+} \PR{\exists k \in [m], \nrm{\sum_{t=1}^{k}\Y_t}_2 > \sqrt{2 \sigma_{i}^2 \eps} + \frac{2}{3}\eps \text{ and } \nrm{\mathbf{\Sigma}_k}_2 \leq \sigma_{i}^2 }.
\end{split}
\end{equation}
Applying Theorem~\ref{theorem:matrix-freedman}
deviation size $\sqrt{2 \sigma_i^2 \eps} + \frac{2}{3}\eps$ yields
\begin{equation}
\begin{split}
\PR{\exists k \in [m], \nrm{\sum_{t=1}^{k}\Y_t}_2 > \sqrt{2 \nrm{\mathbf{\Sigma}_k}_2 \eps} + \frac{5}{3}\eps } &\leq \left(1 + \ceil{\ln(2m/\eps)}_+\right) (d_1 + d_2)  e^{-\eps},
\end{split}
\end{equation}
which concludes the proof.
\end{proof}

The following lemma gives an empirical high-confidence
bound for the problem of learning an unknown Markov chain $\mc$ with respect to the
$\nrm{\cdot}_\infty$ norm.

\begin{lemma}
\label{lemma:mc-empirical-learning-infinity-norm}
Let $X_1, \dots, X_m \sim (\mc, \bmu)$ a $d$-state Markov chain and $\estmc$ defined as in Section~\ref{section:pseudo-sg-ub}.
Then, with probability at least
$1 - \delta$, 
\begin{equation}
\begin{split}
\nrm{\estmc - \mc}_\infty \leq 4 \tau_{\delta/d, m} \sqrt{\frac{d}{\Nmin}},
\end{split}
\end{equation}
where $\tau_{\delta, m} = \inf\set{t > 0 : \left(1 + \ceil{\ln(2m/t)}_+\right) (d + 1)  e^{-t} \leq \delta}$, and $\Nmin$ is defined in (\ref{equation:def-random-variable-convenient-notation}).
\end{lemma}
\begin{proof}
For a fixed $i$,
define the row vector sequence $\Y$ by
\begin{equation}
\begin{split}
\Y_{0} = 0, \Y_{t} = \frac{1}{\sqrt{2}} \left[\pred{X_{t-1} = i} (\pred{X_t = j} - \mc(i,j)) \right]_{j}.
\end{split}
\end{equation}
Notice that $\sum_{t=1}^{m} \Y_t = \left[ N_{i j} - N_i \mc(i, j)\right]_j$, and from the Markov property $\E[t-1]{\Y_{t}} = \zero$, so that $\Y_t$ defines a vector valued martingale difference. Let $\W_{\col, m}, \W_{\row, m}$ and $\nrm{\mathbf{\Sigma}_m}_2$ be as defined in Theorem~\ref{theorem:matrix-freedman}. Then
\begin{equation}
\begin{split}
\Y_t \Y_t \trn &= \nrm{\Y_t}_2^2 = \frac{1}{2} \pred{X_{t-1} = i} \left( 1  + \nrm{\mc(i,\cdot)}_2^2 - 2 \mc(i,X_t) \right) \leq \pred{X_{t-1} = i} \\
\left[\Y_t \trn \Y_t\right]_{(j, \ell)} &= \frac{\pred{X_{t-1} = i}}{2} \left[ \pred{X_t = j} - \mc(i,j) \pred{X_t = \ell} - \mc(i, \ell) \right]_{(j, \ell)},
\end{split}
\end{equation}
so that $\W_{\col, m} \leq N_i$.
Further,
$\nrm{\Y_t \trn \Y_t}_2 \leq \sqrt{\nrm{\Y_t \trn \Y_t}_1  \nrm{\Y_t \trn \Y_t}_\infty}$
(H\"older's inequality),
$\nrm{\Y_t \trn \Y_t}_2 \leq \pred{X_{t-1} = i}$,
and
$\nrm{\mathbf{\Sigma}_m}_2 \leq N_i$
(sub-additivity of the norm and Jensen's inequality).

Corollary~\ref{corollary:freedman-matrix-compacted}
and another application of H\"older's inequality, yields, for all $i \in [d]$,
\begin{equation}
\begin{split}
&\PR[\mc, \bmu]{\nrm{\estmc(i, \cdot) - \mc(i, \cdot)}_1 > \sqrt{\frac{2 d \tau_{\delta/d, m}}{N_i}} + \frac{5 \tau_{\delta/d, m} \sqrt{d}}{3 N_i}  } \\
&\leq \PR[\mc, \bmu]{\nrm{\estmc(i, \cdot) - \mc(i, \cdot)}_1 > \left( \sqrt{2 \nrm{\mathbf{\Sigma}_m}_2 \tau_{\delta/d, m}} + \frac{5}{3} \tau_{\delta/d, m} \right) \frac{\sqrt{d}}{N_i} } \\
&\leq \PR[\mc, \bmu]{\nrm{\sum_{t=1}^{m} \Y_t}_2 >  \sqrt{2 \nrm{\mathbf{\Sigma}_m}_2 \tau_{\delta/d, m}} + \frac{5}{3} \tau_{\delta/d, m} } \leq \frac{\delta}{d}.
\end{split}
\end{equation}
Finally, the observation that
\begin{equation}
\sqrt{\frac{2 d \tau_{\delta/d, m}}{\Nmin}} + \frac{5 \tau_{\delta/d, m} \sqrt{d}}{3 \Nmin} \leq 4 \tau_{\delta/d, m} \sqrt{\frac{d}{\Nmin}}
\end{equation}
completes the proof.
\end{proof}

\begin{remark}
  In the discrete distribution learning model
  \citet{DBLP:conf/stoc/KearnsMRRSS94, DBLP:conf/innovations/Waggoner15},
  the minimax complexity for learning a $[d]$-supported distribution up to precision
  $\eps$ with high confidence is of order
  $m = \tilde{\Theta} \left( \frac{d}{\eps^2} \right)$,
  and so the bound
  in Lemma~\ref{lemma:mc-empirical-learning-infinity-norm} is in a sense optimal,
  and can be thought of as an empirical version of the bounds derived in \citet{wolfer2018minimax}.
\end{remark}

\begin{corollary}[to Lemma~\ref{lemma:mc-empirical-learning-infinity-norm}]
\label{corollary:learning-time-reversal-empirical-bounds}
Let $X_1, \dots, X_m \sim (\mc, \bmu)$ a $d$-state Markov chain and $\estmc \rev$ is such that $\estmc \rev (i,j) \eqdef \frac{N_{ j i }}{N_i}$. Then, with probability at least
$1 - \delta$, 
\begin{equation}
\begin{split}
\nrm{\estmc \rev - \mc \rev}_\infty \leq 4 \tau_{\delta/d, m} \sqrt{\frac{d}{\Nmin}},
\end{split}
\end{equation}
where $\mc \rev$ is the time reversal of $\mc$, $\tau_{\delta, m}$ is as in Lemma~\ref{lemma:mc-empirical-learning-infinity-norm}, and $\Nmin$ is defined in \eqref{equation:def-random-variable-convenient-notation}.
\begin{proof}
The only change is to consider the time-reversed martingale
\begin{equation}
\begin{split}
\Y_{0} = 0, \Y_{t} = \frac{1}{\sqrt{2}} \left[\pred{X_{t} = i} (\pred{X_{t-1} = j} - \mc \rev (i,j)) \right]_{j}
\end{split}
\end{equation}
and mimic
the proof of Lemma~\ref{lemma:mc-empirical-learning-infinity-norm}.
This yields the claim with probability at least $1 - \delta$.
\end{proof}
\end{corollary}

The next lemma
provides
a perturbation bound for the stationary distribution of a reversible Markov chain
in terms of the symmetrized estimator of the absolute spectral gap.

\begin{corollary}
\label{corollary:reversible-kappa-control}
Let $X_1, \dots, X_m \sim \mc$ an ergodic reversible Markov chain with stationary distribution $\bpi$ minorized by $\pimin$ and absolute-spectral gap $\asg$, and let $\estmc, \estpi$ the normalized-count estimators for $\mc, \bpi$. Then
\begin{equation}
\begin{split}
\nrm{\estbpi - \bpi}_\infty \leq \frac{C_\mathcal{K}}{\estasg} \ln \left( 2 \sqrt{\frac{2m}{\Nmin}} \right) \left( \nrm{\estmc - \mc}_\infty + \nrm{\estmc \rev - \mc \rev }_\infty \right),
\end{split}
\end{equation}
where $C_\mathcal{K}$ is the universal constant in Lemma~\ref{lemma:control-visits-all-k-skipped-chains}, and $\estasg \eqdef \asg\left( \frac{1}{2} \left( \estmc + \estmc \rev \right) \right) $, the absolute spectral gap of the additive reversiblization of $\estmc$.
\end{corollary}
\begin{proof}
  As $\estbpi$ is also the stationary distribution of $\frac{1}{2} \left( \estmc + \estmc \rev \right)$
  we have the
  the perturbation bound
  $\nrm{\estbpi - \bpi}_\infty \leq \hat{\kappa} \nrm{\frac{1}{2} \left( \estmc + \estmc \rev \right) - \mc}_\infty$,
  where
  $\hat{\kappa} = \kappa \left( \frac{1}{2} \left( \estmc + \estmc \rev \right) \right)$.
  By reversibility of $\mc$ and norm sub-additivity,
\begin{equation}
\begin{split}
\nrm{\estbpi - \bpi}_\infty \leq \hat{\kappa} \frac{1}{2} \left( \nrm{  \estmc - \mc }_\infty + \nrm{  \estmc \rev - \mc \rev }_\infty \right).
\end{split}
\end{equation}
The proof is concluded by invoking
Lemmas~\ref{lemma:non-reversible-kappa-control} and~\ref{lemma:pssg-reversible}.
\end{proof}

\section{Results from the literature}

\begin{lemma}[from \citealp{kazakos1978bhattacharyya}]
\label{lemma:mc-hellinger-recursive}
Let $(\mc_0, \bmu_0)$ and $(\mc_1, \bmu_1)$ two $d$-state Markov chains, and define
\beq
\left[ \bmu_0, \bmu_1  \right]_{\surd} \eqdef \left[ \sqrt{\bmu_0(i) \bmu_1(i)}  \right]_{i \in [d]} \text{ and }
\left[ \mc_0, \mc_1  \right]_{\surd}
\eqdef \left[ \sqrt{\mc_0(i,j) \mc_1(i,j)}  \right]_{(i,j) \in [d]^2}.
\eeq
Then for $\X = (X_1, \dots, X_m), \Y = (Y_1, \dots, Y_m)$ two trajectories of length $m$ sampled respectively from two $(\mc_0, \bmu_0)$ and $(\mc_1, \bmu_1)$, it holds that
\beq
1 - H^2(\X, \Y) = \left[ \bmu_0, \bmu_1  \right]_{\surd} \trn \cdot \left(  \left[ \mc_0, \mc_1  \right]_{\surd}\right)^m \cdot \unit.
\eeq
\end{lemma}

\begin{theorem}[Rectangular Matrix Freedman, {\citet[Corollary~1.3]{tropp2011freedman}}]
\label{theorem:matrix-freedman}
Consider a matrix martingale $\set{\X_t : t = 0, 1, 2, \dots}$ whose values are matrices with dimension $d_1 \times d_2$, and let $\set{\Y_t : t = 1, 2, 3, \dots}$ be the difference sequence. Assume that the difference sequence is uniformly bounded with respect to the spectral norm, i.e. $\exists R > 0$,
$$\nrm{\Y_t}_2 \leq R \text{ almost surely for }  t = 1,2, \dots$$
Define two predictable quadratic variation processes for this martingale:
\begin{equation}
\begin{split}
\W_{\col, m} \eqdef \sum_{t=1}^{m} \E[t-1]{ \Y_t \Y_t \trn} \text{ and } \W_{\row, m} \eqdef \sum_{t=1}^{m} \E[t-1]{ \Y_t \trn \Y_t} \text{ for } m = 1,2,3, \dots , \\
\end{split}
\end{equation}
and write $\nrm{\mathbf{\Sigma}_m}_2 = \max \set{ \nrm{\W_{\col, m}}_2, \nrm{\W_{\row, m}}_2 }$. Then, for all $\eps \geq 0$ and $\sigma^2 > 0$,
\begin{equation}
\begin{split}
\PR{\max_{m \geq 0} \nrm{\sum_{t=1}^{m}\Y_t}_2 > \eps \text{ and } \nrm{\mathbf{\Sigma}_m}_2 \leq \sigma^2} &\leq (d_1 + d_2) \expo{-\frac{\eps^2/2}{\sigma^2 + R\eps/3}} 
\end{split}
\end{equation}
\end{theorem}

\end{document}